\documentclass[11pt]{article}

\usepackage[margin=0.96in]{geometry}
\usepackage{amsmath,amssymb,amsthm,mathtools}
\usepackage{mathrsfs}
\usepackage{xcolor}
\usepackage{microtype}
\usepackage[numbers,sort&compress]{natbib}
\usepackage[colorlinks=true,linkcolor=blue,citecolor=blue,urlcolor=blue]{hyperref}
\usepackage{mathtools}\mathtoolsset{showonlyrefs}
\usepackage{indentfirst}
\usepackage{tikz}

\newtheorem{theorem}{Theorem}[section]
\newtheorem{corollary}[theorem]{Corollary}
\newtheorem{proposition}[theorem]{Proposition}
\newtheorem{lemma}[theorem]{Lemma}

\newtheorem{remark}{Remark}
\newtheorem*{conjecture}{Hot Spots Conjecture}

\newcommand{\Sph}{\mathbb S^2}
\newcommand{\Hyp}{\mathbb H^2}
\newcommand{\R}{\mathbb R}
\newcommand{\dist}{\operatorname{dist}}
\newcommand{\diam}{\operatorname{diam}}
\newcommand{\Crit}{\operatorname{Crit}}
\newcommand{\Area}{\operatorname{Area}}

\newcommand{\calA}{\mathcal A}
\newcommand{\calC}{\mathcal C}

\numberwithin{equation}{section}

\title{Quantitative characterization of deviations from the hot spots conjecture on convex domains in two-dimensional space forms\footnote{\footnotesize The work is supported by National Natural Science Foundation of China (No.12001276, No.12326303, No.12431018).}}

\author{Haiyun Deng$^{1}$, Xuyong Jiang$^{2}$, Xiaoping Yang$^{3}$\footnote{\footnotesize E-mail: hydeng@nau.edu.cn(H. Deng), jiangxy@cczu.edu.cn(X. Jiang), xpyang@nju.edu.cn(X. Yang)} \\[12pt]
	\emph {\scriptsize $^{1}$Department of Applied Mathematics, Nanjing Audit University, Nanjing, 211815, China;}\\
	\emph {\scriptsize $^{2}$Department of Mathematics, Changzhou University, Changzhou, 213164, China;}\\
	\emph {\scriptsize $^{3}$Department of Mathematics, Nanjing University, Nanjing, 210093, China} }
\date{}

\begin{document}
	
	\maketitle
	
	\begin{abstract}
    In this paper, we study critical points of second Neumann eigenfunctions on convex domains in two-dimensional space forms from three complementary perspectives: spectral and geometric criteria for the absence of interior critical points, quantitative localization of possible critical points, and explicit bounds for the hot spots constant. Precisely, we first establish monotonicity-radius localization principles in $\mathbb S^2$ and $\mathbb H^2$. As a consequence, we obtain a unified diameter criterion for convex domains in these two space forms: if
$\mu_2(\Omega)D^2\le j_{1,1}^2,$
then every second Neumann eigenfunction on $\Omega$ has no interior critical points.
		Moreover,  when interior critical points may exist, we derive explicit quantitative
		location restrictions in terms of the domain  diameters in $\mathbb{S}^{2}$ and $\mathbb{H}^{2}$.
		Finally, we develop an analytic approach to  study the  \emph{hot spots constant} $\mathfrak{C}(\Omega)$ on convex domains. For planar convex domains, we improve the Euclidean upper bound to
			$\mathfrak{C}(\Omega)<1.48$. We further obtain the corresponding hot spots constants for
		convex domains in non-Euclidean space forms, that is, $\mathfrak{C}(\Omega) < 4$ for $\Omega\subset\mathbb{S}^{2}$ contained in a hemisphere;
			$\mathfrak{C}(\Omega) < 11.2$ for $\Omega\subset\mathbb{H}^{2}$. Our proofs combine the properties of Bessel and Legendre functions, eigenvalue estimates, and Green's identity.
		Our results quantitatively measure ``how wrong'' the  \emph{hot spots conjecture} can be. 
	\end{abstract}

	{\bf Keywords:} Hot spots; Neumann eigenfunctions; Convex domain; Space forms
	
	{{\bf 2020 Mathematics Subject Classification.} Primary: 35B38, 58J50; Secondary: 35J25, 35J05}

	\section{Introduction}
	
	In this paper, we investigate the following Neumann eigenvalue problem
	\begin{equation}\label{1.1}\begin{array}{l}
			\left\{
			\begin{array}{l}
				\triangle u+\mu u=0~~\mbox{in}~~\Omega,\\
				\frac{\partial u}{\partial n}=0 ~~\mbox{on}~~\partial\Omega,
			\end{array}
			\right.
	\end{array}\end{equation}
	where $\Omega$ is a bounded domain in a two-dimensional simply connected space form and $n$ denotes the unit outward normal vector to $\partial\Omega.$ The corresponding eigenvalues $\mu_j(\Omega)$ of \eqref{1.1} are nonnegative and can be arranged as
	$$0=\mu_1<\mu_2\le\mu_3\le\cdots\to\infty.$$

	The hot spots conjecture, formulated by J. Rauch in 1974 \cite{Rauch1975},
	asserts that for a given bounded Lipschitz domain $\Omega\subset \R^n$,
	solutions of a Neumann heat equation with generic initial conditions have
	their extrema tending toward $\partial\Omega$ as time tends to infinity.
	In spectral terms, this is equivalent to requiring that the second Neumann eigenfunction $u$ of  the Laplacian attains its maximum and minimum (the ``hot spots'') exclusively on the boundary of the domain. For general domains in $\mathbb{R}^n$, the conjecture is not true. The conjecture has been proven false for certain planar multiply connected domains \cite{Burdzy1999Ann,Burdzy2005Duke} and convex domains in higher dimensions (for sufficiently large dimension) \cite{deDios2024Arxiv}. In light of these counterexamples, the conjecture is now commonly stated as follows.
	
	\begin{conjecture}\label{conjecture1}
		The second Neumann eigenfunction attains its global maximum and minimum only on the boundary
		of the domain $\Omega$ provided that $\Omega$ is a convex domain in $\mathbb{R}^2$, or a simply connected planar domain.
	\end{conjecture}

	At the same time, many positive results have been established for the hot spots conjecture in special classes of domains. The first positive result was due to Kawohl for cylindrical domains (see \cite[Corollary 2.15]{Kawohl1985} ); he further noted that the conjecture holds for specific domains including parallelepipeds, balls, and annuli (see \cite[page 46]{Kawohl1985}). In 1999, Ba\~{n}uelos and Burdzy made the first major breakthrough using probabilistic methods in \cite{Banuelos1999JFA}. For the simple second Neumann eigenvalue, they coupled Brownian motion with the eigenfunction via deformation of initial conditions, and combined this with eigenvalue estimates, to establish the conjecture for a class of planar convex domains, including certain special lip domains—where a “lip domain” is defined as a planar region bounded between the graphs of two functions with Lipschitz constants at most 1. In 2000, Jerison and Nadirashvili introduced the method of domain deformation in \cite{Jerison2000JAMS}, proving the hot spots conjecture via a continuity method on planar domains with two axes of symmetry. This provided a systematic analytical approach to the hot spots conjecture.  Another significant contribution came from Atar and Burdzy in \cite{Atar2004JAMS}, who established both the simplicity of the second Neumann eigenvalue and the validity of the hot spots conjecture for lip domains. 
	
	A decade later, in 2012, a major online collaborative initiative—dubbed ``Polymath 7" \cite{Polymath2012}—was formally launched with the explicit aim of resolving the conjecture for acute triangles. Among the project's most notable outputs was a revealing numerical observation: the extremal points appeared to be located precisely at the two endpoints of the triangle's longest side. The conjecture for obtuse or right triangles was addressed in \cite{Atar2004JAMS, Banuelos1999JFA} and, more recently, in \cite{Rohleder2021Arxiv}; the isosceles case follows by combining \cite{Laugesen2010JDE} and \cite{Miyamoto2009JMP}. In 2015, Siudeja \cite{Siudeja2015MathZ} extended the result to acute triangles with the smallest angle at most $\pi/6$, building on the ideas of Miyamoto  \cite{Miyamoto2009JMP}. For non-equilateral triangles, the simplicity of the second Neumann eigenvalue has been established for obtuse and right triangles in \cite{Atar2004JAMS} and for acute triangles in \cite{Siudeja2015MathZ}. A significant recent advance is due to Judge and Mondal \cite{Judge2020Ann,Judge2022Ann}, who showed that the second Neumann eigenfunction on a triangle has no interior critical points, implying that the global extrema must lie entirely on the boundary. In 2026, Chen, Gui, and Yao \cite{Chen2026Invent} proved that in any triangle, the second Neumann eigenfunction $u$ has at most one non-vertex critical point, and $u$ is monotonic in a specific direction. Recently, Deng et al. \cite{Deng2026Arxiv} studied the symmetry properties and non-vertex critical points of the second Neumann eigenfunction $u$, as well as the multiplicity of the corresponding eigenvalue, on isosceles trapezoids and kites. Consequently, they proved that the hot spots conjecture holds for all simple-eigenvalue domains they considered. For related results, see \cite{Chen2019JMPA,Chen2026AMPA,Deng2025AMPA,Hatcher2025SIAMJMA,Hatcher2024PAMS,Judge2025AMQ,Judge2022CPDE,Kennedy2026TAMS,Langford2023MathAnn} and references therein.
	
	It remains a major open problem for planar convex domains, despite many special cases having been solved.
	For bounded convex domains in the Euclidean plane $\R^2$, Miyamoto \cite{Miyamoto2009JMP} proved that the hot spots conjecture holds for planar convex domains under the condition of $\mu_2 D^2 \le j_{1,1}^2$, as well as for a certain class of planar convex domains.
	Recently, this approach has been generalized in two distinct directions. First, Rohleder \cite{Rohleder2026CAOT} established a quantitative location restriction for planar convex domains, showing that any interior critical point of the second Neumann eigenfunction must be sufficiently close to the boundary, relative to the domain's diameter, thereby creating a ``hot spots free'' subregion. Steinerberger showed that the maximum and minimum are attained close to the points achieving maximum distance by using probabilistic methods in \cite{Steinerberger2020CPDE}. Second, for spaces with constant curvature, Hatcher \cite{Hatcher2026Arxiv} successfully adapted Miyamoto's method to convex domains in the hyperbolic plane $\mathbb{H}^2$. Hatcher showed that if the second Neumann eigenvalue $\mu_2(\Omega) \leq 1/4$ for a convex domain $\Omega\subset\mathbb{H}^2$, then the second Neumann eigenfunction has no interior critical points.

	It is worth noting that, in view of the counterexamples mentioned above, the interior extremum can exceed the boundary extremum in magnitude. To quantify the maximum possible deviation between the interior and boundary extrema, Steinerberger \cite{Steinerberger2022RMI}  (see also \cite{Kleefeld2021ACM}) introduced the \emph{hot spots constant}
	\begin{equation}\label{eq:HS-constant}
		\mathfrak{C}(\Omega):=\sup\left\{\frac{||u||_{L^\infty(\Omega)}}{||u||_{L^\infty(\partial\Omega)}} \Big| u\text{ is a second Neumann eigenfunction}\right\}.
	\end{equation}
	This constant answers the question: ``How `wrong' can the Hot Spots conjecture be?''.  Hot spots conjecture implies that $\mathfrak{C}(\Omega)=1$ for planar convex domains $\Omega$.
	Thus, establishing explicit upper bounds for $\mathfrak{C}(\Omega)$ is of independent interest even when the exact value $1$ remains out of reach.
	
	Steinerberger \cite{Steinerberger2022RMI} proved that for any bounded connected domain in $\mathbb{R}^n$ with smooth boundary, $\mathfrak{C}(\Omega)\le 58.35$, and the bound improves as the dimension grows, asymptotically approaching $\sqrt{e^e}\approx 3.89$. Subsequently, Mariano, Panzo and Wang \cite{Mariano2023PA} extended this framework to bounded Lipschitz domains and significantly improved the bounds. In dimension $2$, $\mathfrak{C}(\Omega)\le 5.1043$, while for asymptotically large dimension the upper bound can be improved to  $\sqrt{e}\approx 1.6487$. Both of these results rely heavily on probabilistic machinery, including heat kernel estimates and reflected Brownian motion. Recently, de Dios Pont, Hsu, and Taylor \cite{deDios2025Arxiv} identified the sharp extremal value $3.1642$ for connected Lipschitz domains in dimension two and showed that the corresponding dimension-dependent extremal ratio converges to $\sqrt e$ as $d\to\infty$.

	\subsection{Main results}
	The Euclidean problem has been extensively studied, whereas the curved settings require additional control of radial solutions to the Helmholtz equation.
	Our contributions concern monotonicity-radius localization in the spherical and hyperbolic settings,
	explicit geometric consequences, and explicit hot spots constants on the three two-dimensional space forms.
    
	Throughout this paper, we assume that $\Omega$ is convex in the space forms,
		where a domain $\Omega$ is convex if any two points in $\Omega$ can be joined by a length-minimizing geodesic contained in $\Omega$.
		We also assume that the boundary $\partial \Omega$ is piecewise $C^{1,\alpha}$.
		All Neumann conditions and normal derivatives at nonsmooth boundary points are interpreted in the weak trace sense, and statements involving $\partial_n$ on $\partial\Omega$ are understood at almost every regular boundary point.
		We now state our first main result, which gives a unified diameter criterion for the absence of interior critical points on convex domains in the two curved space forms $\mathbb S^2$ and $\mathbb H^2$.
	
	\begin{theorem}\label{thm:intro-critical-free-space-forms}
		Let $u$ be a second Neumann eigenfunction to \eqref{1.1} on a domain $\Omega$, $\mu=\mu_2(\Omega)$, and $D$ be the diameter of $\Omega$. If $\mu D^2\leq j_{1,1}^2$, and $\Omega$ satisfies one of the following conditions:
		\begin{enumerate}
			\item[(S)] $\Omega\subset\Sph$ is a convex domain contained in a hemisphere,
			\item[(H)] $\Omega\subset\Hyp$ is a bounded convex domain,
		\end{enumerate}
		then $u$ has no critical point in $\Omega$, where $j_{1,1}$ is the first positive root of the Bessel function $J_1$.
	\end{theorem}
	
	When the condition that there are no critical points is not met, we establish a universal location restriction for potential interior critical points relative to the domain diameters in $\Sph$ and $\Hyp$.

	\begin{theorem}\label{thm:intro-location-space-forms}
		Let $u$ be a second Neumann eigenfunction to \eqref{1.1} on a convex domain $\Omega$. Let 
		$D$ be the diameter of $\Omega$. If $p\in\Omega$ is an interior critical point of $u$, then
		\begin{enumerate}
			\item[(S)] For $\Omega\subset\Sph$ contained in a hemisphere, there exists $q\in\partial\Omega$ such that
			\begin{align} 
				d_{\mathbb S^2}(p,q)\geq \frac{j_{1,1}}{\sqrt{4j_{0,1}^2-D^2/3}} D \ge  \frac{j_{1,1}}{2j_{0,1}} D \approx 0.7967 D.
			\end{align}
			
			\item[(H)] For $\Omega\subset\Hyp$, there exists $q\in\partial\Omega$ such that
			\begin{align} 
				d_{\Hyp}(p,q) \geq \frac{j_{1,1}}{\sqrt{ 4j_{0,1}^2+D^2/3}}  D.
			\end{align}
		\end{enumerate}
	\end{theorem}
	Finally, we investigate the hot spots constant $\mathfrak{C}(\Omega)$ for convex domains. We develop an analytic approach based on Green's identity and singular fundamental solutions of the Helmholtz equation. This method not only extends quantitative estimates for the hot spots constant to non-Euclidean space forms, but also sharpens the general upper bound $3.1642$ for planar connected Lipschitz domains to
$\mathfrak C(\Omega)<1.48$ for planar convex domains. Compared with the existing bounds for Lipschitz domains \cite{Steinerberger2022RMI,Mariano2023PA,deDios2025Arxiv}, our estimate quantitatively demonstrates that convexity significantly restricts the possible deviation from the hot spots conjecture.
	
	\begin{theorem}\label{thm:intro-hot-spots-constant}
		Let $\Omega$ be a bounded convex domain in a two-dimensional space form.
		\begin{enumerate}
			\item[(E)] For $\Omega\subset\R^2$, we have
				\[
				\mathfrak{C}(\Omega)
				\le \frac{\pi}{2}\left(\frac{7}{10}
				+\frac{(j'_{1,1})^2}{14}\right)
				<1.48,
				\]
				where $j'_{1,1}$ is the first positive root of $J_1'$.
				
				\item[(S)] For $\Omega\subset\Sph$ contained in a hemisphere,  then $\mathfrak{C}(\Omega) < 4$.
				
				\item[(H)] For $\Omega\subset\Hyp$, then $\mathfrak{C}(\Omega) <11.2$.
		\end{enumerate}
	\end{theorem}

	\subsection{Sketch of the proofs of the main theorems}
	For the sake of clarity, we now explain the key ideas that are used to prove the main results.
	
	(1) {\bf Main ideas of the proof of Theorems \ref{thm:intro-critical-free-space-forms} and \ref{thm:intro-location-space-forms}}:  We focus primarily on the two-dimensional sphere and hyperbolic plane. In Euclidean space, scale invariance allows the radial Helmholtz equation to be reduced to a universal Bessel equation through the variable $z=\sqrt{\mu_2(\Omega)}r$, where $r$ is the distance function, so that the monotonicity radius can be easily determined. 
	In curved space forms, curvature breaks this scaling reduction and the monotonicity radius depends nontrivially on both the eigenvalue  and the geometry of domain $\Omega$.
	One of the main challenges is therefore to determine the monotonicity radius of the corresponding radial Helmholtz solutions and to derive quantitative estimates for it.
	The main strategy can be divided into three steps.
	
	(a) {\bf  Construct the comparison function}: Let $\Phi_\mu$ denote a radial solution of the Helmholtz equation
	\(-\Delta \Phi=\mu\Phi\) in the corresponding space form,
	(i.e., $j_\mu(r)$ on $\Sph$, and $k_\mu(r)$ on $\Hyp$)
	and we construct the following function:
	\[
	w(q)=u(p)\Phi_\mu(d(p,q))-u(q).
	\]
	Then \(w\) satisfies the same eigenvalue equation as \(u\), and \(p\) is a
	critical nodal point of \(w\). We show that the local structure of nodal sets rules out closed interior
	nodal loops and yields two positive nodal domains of \(w\) meeting
	\(\partial\Omega\).
	
	(b) {\bf Estimate the monotonicity radius of $\Phi_\mu$ and $\mu_2(\Omega)$}:
	Using Cusa–Huygens type inequalities and the eigenvalues of the Sturm–Liouville problem, we derive a lower bound for the monotonicity radius of $\Phi_\mu$ in terms of $\mu$. We also obtain an upper bound for the second Neumann eigenvalue in terms of the diameter.
	
	(c) {\bf Variational characterization}: Assume that $p\in\Omega$ is an interior critical point of a second Neumann
	eigenfunction $u$. Let $R_p=\max_{q\in\partial\Omega} d(p,q)$.
	If \(R_p\) is smaller than the monotonicity radius, 
	then the two positive nodal domains yield two disjoint test functions; combined with convexity and the variational characterization of $\mu_2$, this leads to a contradiction.

	(2) {\bf Main ideas of the proof of Theorem \ref{thm:intro-hot-spots-constant}}: In contrast to the approaches of Mariano--Panzo--Wang \cite{Mariano2023PA} and Steinerberger \cite{Steinerberger2022RMI}, which are based on the heat kernel and reflected Brownian motion, our method is analytic and uses explicit certified special function estimates.
	The main strategy can be divided into two steps.

	(a) {\bf A unified Green-identity reduction}:
	For each space form, we choose a singular radial fundamental solution $\Gamma$ of the Helmholtz equation
	\[
	\Delta\Gamma+\mu\Gamma=-\delta_p \quad \text{on } \Omega,
	\]
	in the distributional sense, and a regular radial solution $\Phi$ of the corresponding homogeneous equation.
	Applying Green's identity to $(u,\Gamma)$ and $(u,\Phi)$, performing computations in geodesic polar coordinates,
	we reduce the estimate for the hot‑spots constant to bounds on three target functions. More precisely, the three functions are:
	\begin{align}
		&\text{the case of }  \mathbb R^2  :\quad z\bigl(Y_1(z)+cJ_1(z)\bigr),  &&\text{where }z=\sqrt\mu r, \\
		&\text{the case of }  \Sph:\quad \bigl(\mathsf{Q}_\nu'(\cos r)+c\mathsf{P}_\nu'(\cos r)\bigr)\sin^2 r, &&\text{where }\nu(\nu+1)=\mu, \\
		&\text{the case of }  \Hyp:\quad \bigl(\mathcal Q_\nu'(\cosh r)+c\mathcal P_\nu'(\cosh r)\bigr)\sinh^2 r, &&\text{where }\nu(\nu+1)=-\mu.
	\end{align}

	(b) {\bf Estimate the target functions}:
	In $\mathbb R^2$, 
	scale invariance reduces the target function to a combination of
	$J_1$ and $Y_1$ expressed in terms of one scale-invariant variable $z=\sqrt\mu r$.
		Using the sharp isoperimetric inequality for the Neumann eigenvalue and the area formula,
		we convert the estimation of the target function into a one-dimensional upper bound problem.
		This idea is formulated in Theorem \ref{thm:hot-spots-constant}. For the explicit choice \(c=1/5\), Lemma \ref{lem:euclidean-explicit-majorant} establishes the bound
		$$
		\left|z\left(Y_1(z)+\frac15J_1(z)\right)\right|
		\le
		\frac7{10}+\frac{z^2}{14}.
		$$
		Combining this quadratic majorant with the preceding reduction, we obtain the universal constant
		stated in the theorem.

	In $\Sph$, the curvature term prevents $\mu$ and $r$ from being combined exactly into one variable. Our proof is divided into three cases based on the magnitude of the eigenvalues. For small eigenvalues, the global monotonicity recorded in Remark \ref{rem:spherical-low-mu-recovery}
		gives the conclusion directly;
		For intermediate eigenvalues,
		we perform energy estimates, use the energy function to obtain a quadratic control of the target function
		with respect to the parameter $c$, and then optimize over the parameter $c$;
	For large eigenvalues, the target function can be regarded as a small perturbation of the equation 
	\[
	G''-\frac1tG'+G=0,
	\]
	whose fundamental solutions are $U(t)=tJ_1(t)$ and $V(t)=-\frac\pi2tY_1(t)$.
	By choosing a suitable constant $c_\nu$, we set $H_\nu (t)=V(t)+c_\nu U(t)$
	to approximate the target function.
	The estimate for the target function is then achieved via estimates on $H_\nu$.
	
	In $\Hyp$, the negative curvature creates a further difficulty: $\cosh r$ and $\sinh r$ grow exponentially. We therefore
	divide the proof into geometric and spectral regimes. For sufficiently large area, a known hot-spots result rules out interior critical points, and hence $\mathfrak C(\Omega)=1$.  It remains to treat domains with uniformly bounded area.
	Then, similarly to the case of $\mathbb{S}^2$, we proceed by splitting into three regimes according to the magnitude of the eigenvalues.

	\subsection{Notations and outline of the paper}
	
	Throughout this paper, we adopt the following notation for special functions. 
	
	$J_n$ and $Y_n$ denote the Bessel functions of the first and second kind of order $n$, respectively. 
	$j_{n,k}$ is the $k$-th positive root of $J_n$, and $y_{n,k}$ is the $k$-th positive root of $Y_n$,
	where $n\ge 0$ and $k\ge 1$ are integers. 
	
	For Legendre functions, we carefully distinguish their domains of definition based on the geometric setting.
	$\mathsf{P}_\nu$  and $\mathsf{Q}_\nu$ denote the Ferrers functions
	(Legendre functions on the cut) of the first and second kind of degree $\nu$,
	defined on the interval $(-1, 1)$, which naturally arise in the spherical case.
	In contrast, in the hyperbolic case, we use ${P}_\nu$ and ${Q}_\nu$ to denote Legendre functions of the first and second kind of degree $\nu$, defined on $\mathbb{C} \backslash (-\infty, 1]$. If $\nu$ is real, then ${P}_\nu$ and ${Q}_\nu$ are real on $(1,\infty)$.
	If $\nu=-\frac{1}{2}+i\rho$, where $\rho$ is real,
	the corresponding Legendre functions are also called conical functions.
	In this case, ${P}_\nu$ is real on $(1,\infty)$, whereas $ Q_\nu$ is generally complex there. To work with real-valued radial solutions in the hyperbolic case, we set $\mathcal{P}_\nu(x):=P_\nu(x)$ and $\mathcal{Q}_\nu(x):=\operatorname{Re}\{Q_\nu(x)\}$ on $(1,\infty)$.
	
	Regarding the exact definitions and properties of these special functions, we refer the reader to the NIST Digital Library of Mathematical Functions \cite{DLMF} and the references therein.

	The paper is organized as follows. In Section \ref{sec:location}, we establish the spherical critical point localization theorem (Theorem \ref{thm:location})
	and its consequences, i.e., Theorem \ref{thm:intro-critical-free-space-forms} (S) and
	Theorem \ref{thm:intro-location-space-forms} (S).
	In Section \ref{sec:hyperbolic-location},
	we provide the location of possible critical points in $\Hyp$. In Section \ref{sec:hot spots constant}, we prove Theorem \ref{thm:intro-hot-spots-constant} by  a self-contained analytic argument based on Green's identity and special function estimates.

	\section{Location of possible critical points in $\mathbb S^2$}\label{sec:location}

	The next lemmas describe the topological structure of the nodal set of a second Neumann
	eigenfunction and rule out nodal loops enclosing an interior nodal domain.
	\begin{lemma}\label{lem:no-loop}
		Let $\Omega$ be a domain on $\mathbb S^2$ and  $\mu\leq\lambda_1(\Omega)$.  If a nontrivial function $w$ satisfies
		\[
		-\Delta w=\mu w\quad\text{in }\Omega,
		\]
		then the nodal set $\{w=0\}$ has no loop that encloses a nodal domain compactly contained in $\Omega$.
	\end{lemma}
	
	\begin{proof}
		Suppose that such a loop exists.  It encloses a nodal domain $\bar{D}\subset\Omega$ on which $w$ has one sign and $w=0$ on $\partial D$.  Thus $w|_D$ is the first Dirichlet eigenfunction on $D$, and
		\[
		\mu=\lambda_1(D).
		\]
		Since $D\subset\Omega$, by the monotonicity of the first Dirichlet eigenvalue with respect to the domain, we know that
		\[
		\lambda_1(D)>\lambda_1(\Omega).
		\]
		Hence
		\[
		\mu=\lambda_1(D)>\lambda_1(\Omega)\geq\mu,
		\]
		which is a contradiction.
	\end{proof}

	Applying the preceding nonexistence result for nodal loops to the second Neumann eigenfunction,
	we can rule out critical nodal points on its nodal set. 
	\begin{lemma}\label{lem:no-nodal-critical}
		Let $\Omega\subset\mathbb S^2$ be a convex domain contained in a hemisphere, and $u$ be a second Neumann eigenfunction with eigenvalue $\mu=\mu_2(\Omega)$.  Then $u$ has no interior point $p$ such that
		\[
		u(p)=0,
		\qquad
		\nabla u(p)=0.
		\]
	\end{lemma}
	\begin{proof}
		It is well known that $\mu=\mu_2(\Omega)\le \lambda_1(\Omega)$
		for convex domains in $\Sph$ (see \cite[Theorem 3.1]{AshbaughLevine1997}).
		Hence Lemma \ref{lem:no-loop} applies to $w=u$, and the nodal set of $u$ contains no closed
		loop enclosing a nodal domain compactly contained in $\Omega$.
		
		Suppose by the contradiction that there exists
		$p\in\Omega$ such that
		\[
		u(p)=0,
		\qquad
		\nabla u(p)=0.
		\]
		By the standard local structure theorem for nodal sets
		\cite{Cheng1976CMH,Hartman1953AJM}, at least four nodal arcs
		emanate from $p$, and the signs of $u$ alternate in the sectors
		between consecutive arcs.
		
		Since $\Omega$ is convex, it is simply connected. Furthermore, Lemma~\ref{lem:no-loop} rules out any closed components of the nodal set. Hence, continuing the nodal arcs emanating from $p$ through the nodal set forces them all to terminate at $\partial\Omega$. By a standard planar topological argument, the existence of at least four such branches originating from an interior point divides $\Omega$ into no fewer than four nodal domains. However, this contradicts Courant's nodal domain theorem, which states that a second Neumann eigenfunction can have at most two nodal domains. Consequently, no such interior critical nodal point can exist.
	\end{proof}

	\begin{proposition}\label{prop:radial-criterion}
		Let $\Omega\subset\mathbb S^2$ be a convex domain contained in a hemisphere, and $u$ be the second Neumann eigenfunction with eigenvalue $\mu=\mu_2(\Omega)$.  Suppose that $p\in\Omega$ satisfies
		\begin{equation}\label{eq:local-monotone-assumption}
			j_\mu'(r)<0
			\qquad\text{for every }0<r\leq R_p,
			\qquad
			R_p:=\max_{q\in\partial\Omega}d_{\mathbb S^2}(p,q).
		\end{equation}
		Then $p$ is not a critical point of $u$.
	\end{proposition}
	
	\begin{proof}
		We set up the usual contradiction argument. Assume that $p$ is a critical point of $u$. By Lemma \ref{lem:no-nodal-critical}, $u(p)\neq0$.  Replacing $u$ by $-u$ if necessary, assume $u(p)>0$.
		
		Define
		\begin{align}
			w(q)=u(p)J_{p,\mu}(q)-u(q)=u(p)j_\mu(d_{\mathbb S^2}(p,q))-u(q).
		\end{align}
		Then
		\[ -\Delta w=\mu w\quad\text{in }\Omega. \]
		Moreover,
		\[  w(p)=0,\qquad\nabla w(p)=0, \]
		because $J_{p,\mu}(p)=1$, $\nabla J_{p,\mu}(p)=0$ and $\nabla u(p)=0$.
		We first note that $w\not\equiv0$.
		Indeed, if $w\equiv0$, then $u(q)=u(p)j_\mu(r_p(q))$.
		Hence, by the Neumann boundary condition,
		\begin{align}
			0=\partial_n u =u(p)j_\mu'(r_p)\partial_n r_p\quad\text{a.e. on }\partial\Omega.
		\end{align}
		Since $u(p)\neq0$ and the assumption of \eqref{eq:local-monotone-assumption}, it follows that
		$\partial_n r_p=0$ a.e. on $\partial\Omega$, which contradicts the convexity of $\Omega$. 
		
		By the standard local structure theorem for nodal sets, the nodal set of $w$ has at least
		four nodal arcs meeting at $p$.
		By Lemma \ref{lem:no-loop}, the nodal set of $w$ has no loop.
		Therefore there are at least two positive nodal domains of $w$, denoted by $\Omega_1$ and $\Omega_2$.

		At a point $q\in\partial\Omega$, the Neumann condition for $u$ gives
		\[
		\partial_n w(q)=u(p)\,j_\mu'(r_p(q))\,\partial_n r_p(q).
		\]
		By convexity, we have $\partial_n r_p(q)\geq0$ a.e. on $\partial\Omega$, together with assumption \eqref{eq:local-monotone-assumption}, it follows that
		\begin{equation}\label{eq:normal-sign}
			\partial_n w\leq0\quad\text{ a.e. on }\partial\Omega.
		\end{equation}
		
		Let $\phi_i=w|_{\Omega_i}$, extended by zero to all of $\Omega$, for $i=1,2$.  Since $w>0$ on each $\Omega_i$, integration by parts and \eqref{eq:normal-sign} give
		\begin{equation}\label{eq:rayleigh-components}
			\int_\Omega |\nabla\phi_i|^2
			=\mu\int_\Omega \phi_i^2
			+\int_{\partial\Omega_i\cap\partial\Omega} w\partial_n w 
			\leq \mu\int_\Omega \phi_i^2.
		\end{equation}
		The functions $\phi_1$ and $\phi_2$ have disjoint supports.  Choose constants $\alpha,\beta$, not both zero, such that
		\begin{align}\label{eq:orth-constant}
			\int_\Omega (\alpha\phi_1+\beta\phi_2)=0,
		\end{align}
		and set
		\[
		\phi=\alpha\phi_1+\beta\phi_2.
		\]
		Using \eqref{eq:rayleigh-components} and the disjointness of supports, we obtain
		\[
		\int_\Omega |\nabla\phi|^2
		\leq
		\mu\int_\Omega \phi^2.
		\]
		Since \eqref{eq:orth-constant} implies that $\phi$ is orthogonal to constants,
		the Rayleigh quotient characterization of $\mu=\mu_2(\Omega)$ gives the reverse inequality.
		Hence the equality must hold and $\phi$ must be a second Neumann eigenfunction.
		But $\phi$ vanishes on a nonempty open subset of $\Omega$,
		contradicting the unique continuation theorem. Therefore $p$ cannot be a critical point of $u$.
	\end{proof}

	We now introduce the radial comparison function and its initial interval of monotonicity. For any $\mu>0$, let $j_\mu$ be the unique smooth solution on $(0,\pi)$ of
	\begin{equation}\label{eq:spherical-radial-ode}
		j_\mu''(r)+\cot r\,j_\mu'(r)+\mu j_\mu(r)=0,
		\qquad
		j_\mu(0)=1,
		\qquad
		j_\mu'(0)=0.
	\end{equation}
	For a fixed point $p\in\mathbb S^2$, we set
	\[J_{p,\mu}(q)=j_\mu\bigl(d_{\mathbb S^2}(p,q)\bigr),\]
	where $d_{\mathbb S^2}(p,q)$ denotes the geodesic distance between the points $p$ and $q$ on $\Sph$. On $\mathbb S^2\backslash\{-p\}$, this function satisfies
	\[
	-\Delta J_{p,\mu}=\mu J_{p,\mu},
	\qquad
	J_{p,\mu}(p)=1,
	\qquad
	\nabla J_{p,\mu}(p)=0.
	\]
	By the definition of Ferrers functions, this radial function can also be written as $J_{p,\mu}=\mathsf{P}_\nu(\cos r)$,
	where the positive constant $\nu$ satisfies $\nu(\nu+1)=\mu$.

	For $\mu>0$, we define the monotonicity radius of the function $j_\mu$ by
	\[
	\tau_\mu
	:=\sup\{\rho\in(0,\pi]: j_\mu'(r)<0\text{ for every }0<r<\rho\}.
	\]
	Equivalently, if $j_\mu'$ has a first zero point in $(0,\pi)$, then $\tau_\mu$ is this zero point; otherwise $\tau_\mu=\pi$.
	On the other hand, if $\tau_\mu<\pi$, by using Ferrers functions, since 
	$j_\mu'(r) = -\sin(r) \mathsf{P}_\nu'(\cos r)$, 
	the quantity $\cos(\tau_\mu)$ is precisely the root of $\mathsf{P}_\nu'(r)$ in $(-1,1)$ that is the first root closest  to $1$. For those $\mu$ with $\nu\in \mathbb{N}$, $\mathsf{P}_\nu$ has an explicit expression, and thus specific values can also be given. 
	For example, when $\mu = 6$ and $\nu = 2$, we have 
	$\mathsf P_2 = \frac{1}{2}(3x^2 - 1)$, and in this case $\tau_6 = \pi/2$. When $\mu = 12$ and $\nu = 3$, we have 
	$\mathsf P_3 = \frac{1}{2}(5x^3 - 3x)$, and in this case 
	$\tau_{12} = \arccos \frac{\sqrt{5}}{5}$.

	Let $\Omega\subset\mathbb S^2$ be a convex domain contained in a hemisphere, and let $\mu>0$.  For a fixed point $p\in\Omega$, we set
	\[
	R_p=\max_{q\in\partial\Omega}d_{\mathbb S^2}(p,q),
	\]
	and define
	\[
	\calC_\mu(\Omega)
	:=\{p\in\Omega:R_p<\tau_\mu\},
	\qquad
	\calA_\mu(\Omega)
	:=\{p\in\Omega:R_p\geq\tau_\mu\}.
	\]
	The following theorem shows that interior critical points must be in the region $\calA_\mu(\Omega)$.
	\begin{theorem}\label{thm:location}
		Let $\Omega\subset\mathbb S^2$ be a convex domain contained in a hemisphere, and $u$ be the second Neumann eigenfunction with eigenvalue
		$\mu=\mu_2(\Omega)$.
		Then
		\[
		\Crit(u)\cap\Omega\subset\calA_\mu(\Omega).
		\]
		Equivalently,
		\[
		\nabla u(p)\neq0
		\qquad\text{for every }p\in\calC_\mu(\Omega).
		\]
	\end{theorem}
	
	\begin{proof}
		Let $p\in\calC_\mu(\Omega)$.  By the definition of $\calC_\mu(\Omega)$, the local monotonicity condition \eqref{eq:local-monotone-assumption} holds at $p$.  Then Proposition \ref{prop:radial-criterion} implies that $p$ cannot be a critical point of $u$.  Hence all interior critical points lie in $\calA_\mu(\Omega)$.
	\end{proof}

	\begin{remark}\label{rem:spherical-low-mu-recovery}
		If $\mu\leq2$, then $\tau_\mu=\pi$. Indeed, set $g(r)=-j_\mu'(r)$,
		then \eqref{eq:spherical-radial-ode} is equivalent to
		\begin{equation}\label{eq:sturm}
			-\bigl(\sin r  g'(r)\bigr)'
			+\frac1{\sin r}g(r)
			=\mu\sin r g(r).
		\end{equation}
		Moreover, substituting $\cot r=\frac1r-\frac{r}{3}+O(r^3)$ near $r=0$
		into \eqref{eq:spherical-radial-ode}$,$ we obtain $j_\mu(r)=1-\frac{\mu}{4}r^2+O(r^4)$,
		and hence $g(r)=\frac{\mu}{2}r+O(r^3)>0$ for sufficiently small $r$.
		If $g$ had a first zero $\rho\in(0,\pi)$, then $g$ would be
		the positive first Dirichlet eigenfunction of
		the Sturm--Liouville problem \eqref{eq:sturm} on $(0,\rho)$ with eigenvalue $\mu$.
		On the other hand, $v(r)=\sin r$ is the positive first
		Dirichlet eigenfunction of the same problem on $(0,\pi)$
		with eigenvalue $2$. The strict domain monotonicity of the
		first Dirichlet eigenvalue therefore gives
		\[
		\mu > 2,
		\]
		a contradiction. Hence $j_\mu'(r)<0$ for every $0<r<\pi$,
		and consequently $\tau_\mu=\pi$.
		Since $R_p<\pi$ for every $p\in\Omega$, we have
		$\calA_\mu(\Omega)=\varnothing$. Thus $u$ has no interior critical
		point. 
	\end{remark}

	\begin{remark}\label{rem:spherical-lune-example}
		For every smooth convex domain $\Omega\subset\mathbb S^2$,
		the Lichnerowicz-type Neumann eigenvalue estimate gives $\mu_2(\Omega)\geq2$,
		with equality only for a hemisphere \cite{Escobar1990,Xia1991}.
		Consequently, the preceding remark has no genuinely subcritical application in the smooth category.
		Nevertheless, the endpoint $\mu_2=2$ occurs naturally for nonsmooth convex domains.
		For $0<\alpha<\pi$, let $L_\alpha$ be the spherical lune bounded by two meridians meeting at angle $\alpha$,
		with the north and south poles as its two vertices. In colatitude-longitude coordinates, it is given by
		\[
		L_\alpha=\{(\theta,\varphi):0<\theta<\pi,\ 0<\varphi<\alpha\}.
		\]
		Separation of variables with the Neumann condition on the meridians gives $\mu_2(L_\alpha)=2$,
		with eigenfunction $u(\theta,\varphi)=\cos\theta$. 
		This eigenfunction has no critical point in the interior of
		$L_\alpha$, and its only critical points on
		$\overline{L_\alpha}$ are the two vertices, which belong to
		$\partial L_\alpha$.
		On the other hand,
		by Remark \ref{rem:spherical-low-mu-recovery},
		$\tau_2=\pi$, and for every $p\in L_\alpha$ one has $R_p<\pi$.
		Therefore Theorem \ref{thm:location} implies that $\calA_2(L_\alpha)=\varnothing$,
		and $u$ has no interior critical point. Thus Theorem \ref{thm:location} applies to every such domain
		$L_\alpha$.
		Notice also that $\diam(L_\alpha)=\pi$ and
		\[
		\mu_2(L_\alpha)\diam(L_\alpha)^2=2\pi^2>j_{1,1}^2,
		\]
		so this example is not covered by the diameter criterion in Theorem \ref{thm:intro-critical-free-space-forms}.
	\end{remark}

	We give the following quantitative lower bound for $\tau_\mu$ involving the diameter of the domain.

	\begin{lemma}
		\label{lem:tau-lower-bound}
		If $\tau_\mu < \pi$, we have
		\begin{align} \label{eq.3a}
			\tau_\mu
			\ge
			\frac{j_{1,1}}{\sqrt{\mu}},
		\end{align}
		where \(j_{1,1}\) denotes the first positive root of the Bessel function \(J_1\).
	\end{lemma}
	\begin{proof}
		Set $g(r)=-j_\mu'(r)$.
		Obviously $g(r)>0$ for all sufficiently small $r>0$.
		Substituting $g$ into \eqref{eq:spherical-radial-ode} and rearranging gives
		\begin{align}\label{eq:Sph-sturm-liouville}
			-\bigl(\sin r\, g'(r)\bigr)'+\frac{1}{\sin r}g(r)=\mu\sin r\, g(r).
		\end{align}
		By the definition of $\tau_\mu$, we have
		\begin{align}
			g(r)>0 \quad \text{for }0<r<\tau_\mu,\qquad g(\tau_\mu)=0.
		\end{align}
		Set $\rho=\tau_\mu$. Since $r=0$ is a singular endpoint, we consider the Friedrichs form associated with the Sturm--Liouville equation \eqref{eq:Sph-sturm-liouville}. Let
		\begin{align}
			\mathcal H_\rho^{\mathbb S}=
			\left\{v\in H^1_{\mathrm{loc}}(0,\rho)\,\big|\, v(\rho)=0,\
			\int_0^\rho\left(\sin r\, |v'|^2+\frac{|v|^2}{\sin r}\right)dr<\infty\right\}.
		\end{align}
		Since $g(r)=\dfrac{\mu r}{2}+O(r^3)$ as $r\to 0^+$, we have $g\in\mathcal H_\rho^{\mathbb S}$, which is positive on $(0,\rho)$, vanishes at $\rho$, and realizes the lowest eigenvalue of this singular Sturm--Liouville problem.
		Hence
		\begin{align}
			\mu=
			\inf_{0\ne v\in\mathcal H_\rho^{\mathbb S}}
			\frac{\displaystyle \int_0^\rho \left( \sin r |v'|^2+\frac{v^2}{\sin r}\right) dr}
			{\displaystyle \int_0^\rho \sin r v^2 dr}.
		\end{align}
		
		For $v\in C_0^\infty(0,\rho)$, set $y(r)=(\sin r)^{1/2}v(r)$.
		A direct computation yields
		\begin{align}
			\sin r |v'|^2+\frac{v^2}{\sin r}=|y'|^2-\frac{\cos r}{\sin r}yy'+\frac{4+\cos^2 r}{4\sin^2 r}y^2.
		\end{align}
		Since $y\in C_0^\infty(0,\rho)$, integration by parts gives
		\begin{align}
			\int_0^\rho -\frac{\cos r}{\sin r}yy' dr
			=-\frac12 \int_0^\rho\frac{y^2}{\sin^2 r} dr.
		\end{align}
		Thus
		\begin{align}
			\frac{\displaystyle \int_0^\rho\left(\sin r|v'|^2+\frac{v^2}{\sin r}\right)dr}
			{\displaystyle \int_0^\rho\sin r v^2 dr}
			=\frac{\displaystyle
				\int_0^\rho\left(|y'|^2+\frac{2+\cos^2 r}{4\sin^2 r}y^2\right)dr}
			{\displaystyle \int_0^\rho y^2 dr}.
		\end{align}
		
		The following elementary inequality holds:
		\begin{align}
			\frac{2+\cos^2 r}{4\sin^2 r}\ge\frac{3}{4r^2},\quad 0<r<\pi.
		\end{align}
		Consequently,
		\begin{align}
			\mu\ge
			\inf_{0\ne y\in C_0^\infty(0,\rho)}
			\frac{\displaystyle \int_0^\rho\left(|y'|^2+\frac{3}{4r^2}y^2\right)dr}
			{\displaystyle\int_0^\rho y^2 dr}.
		\end{align}
		
		The infimum on the right‑hand side is the lowest eigenvalue of the Friedrichs realization of
		\[
		-\frac{d^2}{dr^2}+\frac{3}{4r^2}
		\]
		on $(0,\rho)$ with Dirichlet boundary condition at $r=\rho$. Its positive first eigenfunction is
		$y(r)=\sqrt{r} J_1\left(\frac{j_{1,1}}{\rho}r\right)$ with eigenvalue $j_{1,1}^2/\rho^2$.
		Therefore, $\mu\ge j_{1,1}^2/\rho^2$, which implies
		\[
		\tau_\mu\ge \frac{j_{1,1}}{\sqrt{\mu}}.
		\]
		This proves the lemma.
	\end{proof}

	\begin{remark}
		Our estimate \eqref{eq.3a} is asymptotically sharp for large $\mu$, since for sufficiently large $\nu$, we have the asymptotic relation (\cite[(14.15.11)]{DLMF})
		\begin{align}
			\mathsf{P}_\nu(\cos r) \sim \sqrt{\frac{r}{\sin r}}J_0((\nu + \frac12)r)
			\sim J_0((\nu + \frac12)r) \text{  for sufficiently small }r.
		\end{align}
		Therefore 
		$$-(\sin r)\mathsf{P}_\nu'(\cos r) \sim -(\nu + \frac12) J_1((\nu + \frac12) r).$$ 
		Consequently, 
		$$j_\mu'(r) \sim -(\nu + \frac12) J_1((\nu + \frac12) r),$$ 
		so, the first root satisfies 
		$$\tau_\mu \sim \frac{j_{1,1}}{\nu + \frac12} \sim \frac{j_{1,1}}{\sqrt{\mu}}.$$
		For example, $\tau_{12}=\arccos \frac{\sqrt{5}}{5}\approx 1.1071$
		and $\frac{j_{1,1}}{\sqrt{\mu}}\approx 1.1061$.
	\end{remark}
	
	\begin{theorem}\label{thm:sphere-mu-D}
		Let $\Omega\subset\mathbb S^2$ be a convex domain contained in a hemisphere, and
		\begin{align}
			D=\diam_{\mathbb S^2}(\Omega), \quad \mu=\mu_2(\Omega).
		\end{align}
		If $\mu D^2 \le j_{1,1}^2$, then every second Neumann eigenfunction has no critical points in $\Omega$.
	\end{theorem}
	\begin{proof}
		If $\tau_\mu=\pi$, then every second Neumann eigenfunction has no 
		critical points in $\Omega$. If $\tau_\mu<\pi$, 
		by assumption and  Lemma \ref{lem:tau-lower-bound}, we have
		\begin{align}
			\tau_\mu \ge \frac{j_{1,1}}{\sqrt{\mu}}\ge D.
		\end{align}
		Since $R_p<D$ for every $p\in\Omega$, so
		$\calA_\mu(\Omega)=\emptyset$, and Theorem \ref{thm:location} implies that
		every second Neumann eigenfunction has no critical points.
	\end{proof}

	We now give an upper bound for the second Neumann eigenvalue
	involving the diameter. The proof employs a standard test-function argument together with the variational characterization of $\mu_2(\Omega)$.
	
	\begin{lemma}\label{lem:Sph-mu-estimate}
		Let $\Omega\subset\mathbb S^2$ be a convex domain contained in a hemisphere and
		\begin{align}
			D=\diam_{\mathbb S^2}(\Omega).
		\end{align}
		Then we have
		\begin{equation}\label{eq:cheng-kroger-explicit}
			\mu_2(\Omega) \leq \frac{4j_{0,1}^2}{D^2}-\frac{1}{3}.
		\end{equation}
		where $j_{0,1}$ is the first positive root of the Bessel function $J_0$.
	\end{lemma}
	
	\begin{proof}
		For any $\varepsilon>0$, we choose $x,y\in\Omega$ such that
		$d_{\mathbb S^2}(x,y) > D-\varepsilon$ and put
		$r = (D-\varepsilon)/2<\pi/2$.
		The geodesic balls $\mathcal B(x,r)$ and $\mathcal B(y,r)$ are disjoint and their
		centers lie strictly inside of $\Omega$.
		Let $\eta_1,\eta_2$ be the positive radial first Dirichlet eigenfunctions on
		$\mathcal B(x,r)$ and $\mathcal B(y,r)$ respectively, with eigenvalue $\Lambda(r)$.
		
		We extend $\eta_1$ and $\eta_2$ by zero outside of $\mathcal B(x,r)$ and $\mathcal B(y,r)$ respectively. We set
		$\Omega_i = \Omega\cap\mathcal B(x_i,r)$ ($x_1=x$, $x_2=y$).
		Using $-\Delta\eta_i = \Lambda(r)\eta_i$
		and $\partial_n\eta_i(q)=
		\eta_i'(r_{x_i}(q))\,\partial_n r_{x_i}(q)\le 0$,
		we obtain
		\[
		\int_{\Omega_i}|\nabla\eta_i|^2
		= \Lambda(r)\int_{\Omega_i}\eta_i^2
		+ \int_{\partial\Omega_i\cap\partial\Omega}\eta_i\,\partial_n\eta_i
		\le \Lambda(r)\int_{\Omega_i}\eta_i^2 .
		\]
		
		We choose constants $\alpha,\beta$, not both zero, such that
		$\int_\Omega(\alpha\eta_1+\beta\eta_2)=0$.
		By the Rayleigh quotient for $\mu_2(\Omega)$ and disjointness of $\Omega_1,\Omega_2$,
		\[
		\mu_2(\Omega) \le
		\frac{\int_\Omega|\nabla(\alpha\eta_1+\beta\eta_2)|^2}
		{\int_\Omega|\alpha\eta_1+\beta\eta_2|^2}
		\le \Lambda(r).
		\]
		
		Finally, it is well known that
		(see \cite[Theorem 1]{Baginski1990QAM})
		$$\Lambda(r)
		\leq
		\frac{j_{0,1}^2}{r^2}-\frac{1}{3}.$$
		Thus $\mu_2(\Omega)\le 4j_{0,1}^2/(D-\varepsilon)^2-1/3$.
		Let $\varepsilon\to0^+$, we finish the proof.
	\end{proof}

	The lower bound for $\tau_\mu$ and the upper bound for $\mu_2(\Omega)$ now fit together to
	give a universal diameter-scale exclusion region.
	\begin{theorem}\label{cor:diameter-only}
		Let $\Omega\subset\mathbb S^2$ be a convex domain contained in a hemisphere, $u$ be the second Neumann eigenfunction and
		\[
		D=\diam_{\mathbb S^2}(\Omega), \quad \mu=\mu_2(\Omega).
		\]
		If $p\in\Omega$ is a critical point of $u$, then
		there exists a boundary point $q\in\partial\Omega$ such that
		\[
		d_{\mathbb S^2}(p,q)\geq \frac{j_{1,1}}{\sqrt{4j_{0,1}^2-D^2/3}} D \ge  \frac{j_{1,1}}{2j_{0,1}} D \approx 0.7967 D.
		\]
	\end{theorem}
	
	\begin{proof}
		If $\tau_\mu=\pi$, then $u$ does not admit any critical points in $\Omega$,
		which contradicts the assumption.
		By Lemma
		\ref{lem:tau-lower-bound},
		we have $\tau_\mu\geq \frac{j_{1,1}}{\sqrt{\mu}}$.
		On the other hand, Lemma \ref{lem:Sph-mu-estimate} shows that
		$\mu D^2 \leq 4j_{0,1}^2-D^2/3$.
		Hence Theorem \ref{thm:location} implies that
		\begin{align}
			R_p\ge \tau_\mu \ge \frac{j_{1,1}}{\sqrt{4j_{0,1}^2-D^2/3}} D .
		\end{align}
		This completes the proof.
	\end{proof}

	\section{Location of possible critical points in $\Hyp$}\label{sec:hyperbolic-location}

	For $\mu\leq 1/4$, Hatcher \cite{Hatcher2026Arxiv} shows that the second Neumann Laplace eigenfunctions have no
	interior critical points.
	For $\mu>1/4$, the same argument as in the spherical case also yields a location
	restriction for possible interior critical points on a bounded convex domain $\Omega\subset\Hyp$.

	Let $\Omega\subset\Hyp$ be a bounded convex domain.
	For $\mu>0$, let $k_\mu$ be the unique smooth solution of
	\begin{equation}\label{eq:hyperbolic-radial-ode}
		k_\mu''(r)+\coth r\,k_\mu'(r)+\mu k_\mu(r)=0,
		\qquad
		k_\mu(0)=1,
		\qquad
		k_\mu'(0)=0.
	\end{equation}
	For a fixed point $p\in\Hyp$, we set
	\[
	K_{p,\mu}(q)=k_\mu\bigl(d_{\Hyp}(p,q)\bigr).
	\]
	Then
	\[
	-\Delta K_{p,\mu}=\mu K_{p,\mu},
	\qquad
	K_{p,\mu}(p)=1,
	\qquad
	\nabla K_{p,\mu}(p)=0.
	\]
	We define the monotonicity radius of the function $k_\mu$ by
	\[
	\bar\tau_\mu
	:=\sup\{\rho>0:k_\mu'(r)<0\text{ for every }0<r<\rho\}.
	\]
	If $k_\mu$ decreases on all $(0,\infty)$, then $\bar\tau_\mu=\infty$.
	By the definition of Legendre functions,
	$K_{p,\mu}$ can also be written as $K_{p,\mu}=P_\nu(\cosh r)$, where $\nu(\nu+1)=-\mu$. Hence, if $\bar\tau_\mu<\infty$, then $\cosh(\bar\tau_\mu)$ is precisely the first root of $ P_\nu'(z)$ in $(1,\infty)$.

	We may adopt methods analogous to the spherical case to obtain the corresponding results in the hyperbolic space.

	\begin{theorem}\label{thm:hyperbolic-location}
		Let $\Omega\subset\Hyp$ be a bounded convex domain, and let $u$ be the second Neumann eigenfunction with eigenvalue $\mu=\mu_2(\Omega)$.  For $p\in\Omega$, we set
		\begin{align}
			R_p=\max_{q\in\partial\Omega}d_{\Hyp}(p,q),\quad
			\calA_\mu(\Omega)=\{p\in\Omega:R_p\geq\bar\tau_\mu\}.
		\end{align}
		Then
		\[ \Crit(u)\cap\Omega\subset \calA_\mu(\Omega). \]
	\end{theorem}
	\begin{proof}
		The argument is the hyperbolic analogue of the proof of
		Proposition \ref{prop:radial-criterion}, so we only indicate the necessary
		modifications.
		Let $p\in\Omega\setminus\calA_\mu(\Omega)$. Then
		\[
		R_p<\bar\tau_\mu,
		\qquad
		k_\mu'(r)<0
		\quad\text{for }0<r\le R_p.
		\]
		Suppose, to the contrary, that $\nabla u(p)=0$. Since
		$\mu_2(\Omega)\le\lambda_1(\Omega)$ for bounded convex domains in $\Hyp$
		\cite[Proposition 3.14]{Mazzeo1991IMRN}, the arguments of
		Lemmas \ref{lem:no-loop} and~\ref{lem:no-nodal-critical} apply without
		change, and hence $u(p)\neq0$. Replacing $u$ by $-u$ if necessary,
        we may assume $u(p)>0$.
		
		Define
		\[
		w(q)=u(p)k_\mu\bigl(d_{\Hyp}(p,q)\bigr)-u(q).
		\]
		Then
		\[
		-\Delta w=\mu w,\quad
		w(p)=0,\quad\nabla w(p)=0.
		\]
		Moreover, by convexity,
		$\partial_n d_{\Hyp}(p,\cdot)\ge0$ a.e. on $\partial\Omega$, and therefore
		\[
		\partial_n w =u(p)k_\mu'\bigl(d_{\Hyp}(p,\cdot)\bigr)	\partial_n d_{\Hyp}(p,\cdot)
		\le0 \qquad\text{a.e. on }\partial\Omega.
		\]
		Thus all the hypotheses used in the nodal domain and Rayleigh quotient
		argument of Proposition \ref{prop:radial-criterion} are satisfied, with
		$j_\mu$ replaced by $k_\mu$ and $d_{\mathbb S^2}$ replaced by
		$d_{\Hyp}$. The same argument yields a second Neumann eigenfunction
		vanishing on a nonempty open subset of $\Omega$, contradicting unique continuation.
		Hence $\nabla u(p)\neq0$. Since
		$p\in\Omega\setminus\calA_\mu(\Omega)$ was arbitrary, we conclude that
		\[
		\Crit(u)\cap\Omega\subset\calA_\mu(\Omega).
		\]
		This proves the theorem.
	\end{proof}

	\begin{remark}\label{rem:hyperbolic-hatcher-recovery}
		When $\mu\leq1/4$, Lemma 2.1 in \cite{Hatcher2026Arxiv} says precisely that $\bar\tau_\mu=\infty$,
		so $\calA_\mu(\Omega)=\emptyset$, hence Theorem \ref{thm:hyperbolic-location} covers
		Hatcher's results. In fact, if $\bar\tau_\mu \ge  \diam_{\Hyp}(\Omega)$,
		Theorem \ref{thm:hyperbolic-location} implies that there are no critical points in $\Omega$.
	\end{remark}

	\begin{lemma} \label{lem:hyp-tau-lower-bound3}
		We have
		\begin{equation}\label{eq:hyp-tau-bound2}
			\bar\tau_\mu \ge \frac{j_{1,1}}{\sqrt{\mu}}.
		\end{equation}
	\end{lemma}
	
	\begin{proof}
		Set $g(r)=-k_\mu'(r)$.  Then $g(r)>0$ for all sufficiently small $r>0$.
		Substituting $g$ into \eqref{eq:hyperbolic-radial-ode} and rearranging gives
		\begin{align}\label{eq:Hyp-sturm-liouville}
		    -\bigl(\sinh r g'(r)\bigr)' + \frac{1}{\sinh r}g(r) = \mu \sinh r g(r).
		\end{align}

		If $\bar\tau_\mu=\infty$, the bound is immediate. Otherwise, put $\rho=\bar\tau_\mu$.
        As in the spherical case, the endpoint $r=0$ is singular.
        We therefore use the Friedrichs form associated with the
		Sturm–Liouville equation \eqref{eq:Hyp-sturm-liouville}. Let
		\[
		\mathcal H_\rho^{\mathbb H}:=
		\left\{v\in H^1_{\mathrm{loc}}(0,\rho):v(\rho)=0,
		\ \int_0^\rho\left(\sinh r\,|v'|^2+\frac{|v|^2}{\sinh r}\right)dr<\infty\right\}.
		\]
		The radial expansion of \eqref{eq:hyperbolic-radial-ode} gives
		$k_\mu(r)=1-\mu r^2/4+O(r^4)$ and hence $g(r)=\mu r/2+O(r^3)$ near $0$.
		Hence $g\in \mathcal H_\rho^{\mathbb H}$, $g>0$ on $(0,\rho)$ and $g(\rho)=0$. Therefore
		\begin{align}\label{hyp:rayleigh-quotient}
			\mu=\inf_{0\ne v\in\mathcal H_\rho^{\mathbb H}}
			\frac{\displaystyle\int_0^\rho\left(\sinh r |v'|^2+\frac{v^2}{\sinh r}\right)dr}
			{\displaystyle\int_0^\rho\sinh r\,v^2dr}.
		\end{align}
		
		For $v\in C_0^\infty(0,\rho)$,
		set $y(r) = (\sinh r)^{1/2}v(r)$. A Direct computation yields
		\[
		\sinh r |v'|^2 + \frac{v^2}{\sinh r}
		= |y'|^2 - \coth r y y' + \frac{\cosh^2 r + 4}{4 \sinh^2 r} y^2.
		\]
		Integration by parts for the middle term gives
		\[
		\int_0^{\rho} -\coth r y y' dr = -\frac{1}{2} \int_0^{\rho} \frac{1}{\sinh^2 r} y^2 dr.
		\]
		Substituting this back into \eqref{hyp:rayleigh-quotient}, we obtain
		\begin{equation}\label{eq:hyp-rayleigh3}
			\mu = \inf_{0\ne y\in C_0^\infty(0,\rho)} \frac{ \displaystyle \int_0^{\rho} \left( |y'|^2 + \frac{3}{4\sinh^2 r} y^2 + \frac{1}{4} y^2 \right) dr }{ \displaystyle \int_0^{\rho} y^2 dr }.
		\end{equation}

		The following elementary inequality holds
		\[
		\frac{1}{\sinh^2 r} \ge \frac{1}{r^2} - \frac{1}{3}.
		\]
		Substituting this into \eqref{eq:hyp-rayleigh3}, we obtain:
		\begin{align}
			\mu \ge
			\inf_{0\ne y\in C_0^\infty(0,\rho)} \frac{ \displaystyle \int_0^{\rho} \left( |y'|^2 + \frac{3}{4r^2} y^2 \right) dr }{ \displaystyle \int_0^{\rho} y^2 dr }.	
		\end{align}
		Thus $\mu \ge \frac{j_{1,1}^2 }{\rho^2}$, which completes the proof.
	\end{proof}	
	
	\begin{remark} \label{remark-bound}
		In estimating \eqref{eq:hyp-rayleigh3}, noting that the function $r \mapsto \frac{r}{\sinh r}$ is strictly decreasing,
		we can obtain another estimate for $\bar\tau_\mu$. Indeed, if $\bar\tau_\mu\le D$, then we have
		\[ 
		\frac{1}{\sinh^2 r} > \left( \frac{D}{\sinh D} \right)^2 \frac{1}{r^2},
		\]
		hence, from \eqref{eq:hyp-rayleigh3} and $\frac{D}{\sinh D} < 1$, we have
		$$\mu - \frac{1}{4}>\left( \frac{D}{\sinh D} \right)^2 \frac{j_{1,1}^2}{\bar\tau_\mu^2}.$$
		Rearranging this inequality yields
		\begin{align} \label{hyp-tau-bound-other}
			\bar\tau_\mu > \left( \frac{D}{\sinh D} \right) \frac{j_{1,1}}{\sqrt{\mu - 1/4}}.
		\end{align}
		Compared to \eqref{eq:hyp-tau-bound2}, this estimate is better when $\mu$ is close to $1/4$,
		but when $\mu$ is large or $D$ is large, \eqref{eq:hyp-tau-bound2} is better. 
	\end{remark}

	\begin{theorem}\label{thm:hyp} 
		Let $\Omega\subset\Hyp$ be a bounded convex domain and
		\[
		D=\diam_{\Hyp}(\Omega), \quad \mu=\mu_2(\Omega).
		\]
		If $\mu D^2\le j_{1,1}^2 $, then every second Neumann eigenfunction has no critical points in $\Omega$.
	\end{theorem}
	
	\begin{proof}
		By assumption and Lemma \ref{lem:hyp-tau-lower-bound3} , we have
		\begin{align}
			\bar\tau_\mu \ge   \frac{j_{1,1}}{\sqrt{\mu }}\ge D.
		\end{align}
		Therefore  $\bar\tau_\mu \ge D$, and
		the conclusion follows directly from Theorem~\ref{thm:hyperbolic-location}.
	\end{proof}

	By using the estimate from Remark \ref{remark-bound}, we can obtain another result.
	\begin{proposition}
		Let $\Omega\subset\Hyp$ be a bounded convex domain and
		\[
		D=\diam_{\Hyp}(\Omega), \quad \mu=\mu_2(\Omega).
		\]
		If $(\mu-\frac{1}{4})\sinh(D)^2\le j_{1,1}^2 $, then every second Neumann eigenfunction has no critical points in $\Omega$.
	\end{proposition}
	
	\begin{proof}
		%If $\bar\tau_\mu \ge D$, results follows from Remark \ref{rem:hyperbolic-hatcher-recovery} direction. 
		If $\mu \le 1/4$,  the conclusion follows from Remark \ref{rem:hyperbolic-hatcher-recovery}.
		Therefore, we assume that $\mu > 1/4$.
		If $\bar\tau_\mu \le D$, by assumption and \eqref{hyp-tau-bound-other} , we have
		\begin{align}
			\bar\tau_\mu > \left( \frac{D}{\sinh D} \right) \frac{j_{1,1}}{\sqrt{\mu - 1/4}}\ge D,
		\end{align}
		which contradicts the assumption.
		Therefore  $\bar\tau_\mu > D$, and
		the conclusion follows directly from Theorem \ref{thm:hyperbolic-location}.
	\end{proof}
	\begin{remark}
		This result also covers the case $\mu\le 1/4$. Compared with Theorem \ref{thm:hyp}, each has its own advantages.
	\end{remark}

	\begin{lemma}\label{lem:hyperbolic-diameter-upper2}
		Let $\Omega\subset\Hyp$ be a bounded convex domain and 
		\[
		D=\diam_{\Hyp}(\Omega).
		\]
		Then
		\begin{align}
			\mu_2(\Omega)\leq
			\frac{1}{3}+\frac{4j_{0,1}^2}{D^2}.
		\end{align}
	\end{lemma}
	
	\begin{proof}
		We may adopt methods analogous to those in Lemma \ref{lem:Sph-mu-estimate} to get
		$\mu_2(\Omega)\le \Lambda(D/2)$, where $\Lambda(r)$ is the first Dirichlet eigenvalue on a geodesic ball
		with radius $r$. It is well known that (see \cite[Theorem 3.1]{Berge2023AM})
		\begin{align} 
			\Lambda(r)\le \frac{1}{3}+\frac{j_{0,1}^2}{r^2},
		\end{align}
		which completes the proof.
	\end{proof}

	By the lower bound for $\bar\tau_\mu$ and the upper bound for $\mu_2(\Omega)$, we have the following results.
	
	\begin{theorem}\label{cor:hyperbolic-diameter-location3}
		Let $\Omega\subset\Hyp$ be a bounded convex domain, $u$ be the second Neumann eigenfunction, and  $D=\diam_{\Hyp}(\Omega)$. If $p\in\Omega$ is a critical point of $u$, then there exists $q\in\partial\Omega$ such that
		\begin{align}
			d_{\Hyp}(p,q) \geq \frac{j_{1,1}}{\sqrt{ 4j_{0,1}^2+D^2/3}}  D.
		\end{align}
	\end{theorem}
	\begin{proof}
		By Theorem \ref{thm:hyperbolic-location} and Lemma \ref{lem:hyp-tau-lower-bound3}, we have
		$R_p\ge \bar\tau_\mu \ge \frac{j_{1,1}}{\sqrt{\mu}}$.
		Lemma \ref{lem:hyperbolic-diameter-upper2} gives $\mu D^2\le 4j^2_{0,1}+D^2/3$.
		Therefore
		\begin{align}
			R_p\ge \frac{j_{1,1}}{\sqrt{\mu}}\ge \frac{j_{1,1}}{\sqrt{ 4j_{0,1}^2+D^2/3}}  D.
		\end{align}
		This completes the proof.
	\end{proof}
	
	\begin{remark} 
		The coefficient $\frac{j_{1,1}}{\sqrt{ 4j_{0,1}^2+D^2/3}}$ is strictly decreasing with respect to $D$. As $D \to 0$,
		which corresponds to approaching Euclidean space, we have
		\[
		\lim_{D \to 0} \frac{j_{1,1}}{\sqrt{ 4j_{0,1}^2+D^2/3}}   = \frac{j_{1,1}}{2j_{0,1}} \approx 0.7967.
		\]
		This remains a nontrivial estimate when the coefficient is greater than $0.5$, i.e., $\frac{j_{1,1}}{\sqrt{ 4j_{0,1}^2+D^2/3}} >\frac{1}{2}\Leftrightarrow D<\sqrt{12(j_{1,1}^2-j_{0,1}^2)}\approx 10.33$. Otherwise, the conclusion is trivial.
	\end{remark}

	\section{Hot spots constant in space forms}\label{sec:hot spots constant}
	In this section, we use Green’s identity with the singular
	fundamental solution of the Helmholtz equation to establish upper bounds for the hot spots constant
	on bounded convex domains in the two-dimensional space forms.
	Our approach is purely analytic, in contrast to the methods of Mariano-Panzo-Wang \cite{Mariano2023PA} and Steinerberger \cite{Steinerberger2022RMI}, which rely on 
	the heat kernel and the reflected Brownian motion.

	\subsection{Hot spots constant in $\R^2$ }
	
	We begin with the Euclidean case. By combining Green's identity with geometric and spectral bounds for convex planar domains, the estimate of the hot-spots constant can be reduced to a one-dimensional majorization problem for a scale-invariant combination of Bessel functions. This reduction is summarized in the following theorem.
	
	\begin{theorem}\label{thm:hot-spots-constant}
		Let $\Omega\subset\mathbb R^2$ be a bounded convex domain. For
		$c\in\mathbb R$, define
		\begin{align}
			F_c(z)=z\bigl(Y_1(z)+cJ_1(z)\bigr),\qquad 0<z\le 2j_{0,1}.
		\end{align}
		%and extend it continuously to $z=0$ by $F_c(0)=-2/\pi$.
		Define
		\begin{equation}\label{eq:minimax}
			\mathscr M=
			\inf_{\substack{c,\alpha\in\R\\ \beta\ge0}}
			\left\{
			\alpha+\beta (j'_{1,1})^2:
			|F_c(z)|\le \alpha+\beta z^2
			\ \text{for every }z\in (0,2j_{0,1}]
			\right\}.
		\end{equation}
		where $j'_{1,1}\approx 1.841$ is the first positive zero of $J_1'(z)$. 
		Then
		\begin{equation}\label{eq:general-euclidean-bound}
			\mathfrak C(\Omega)\le \frac{\pi}{2}\,\mathscr M .
		\end{equation}
	\end{theorem}
	
	\begin{proof}
		Let $u$ be any fixed second Neumann eigenfunction.
		If the global maximum of $|u|$ is attained on $\partial\Omega$, the desired estimate is immediate. We may therefore assume that it is attained at an interior point $p\in\Omega$. Replacing $u$ by $-u$ if necessary, choose $p\in\Omega$
		such that $M=u(p)>0$. Set $m = \max\limits_{x \in \partial\Omega} |u(x)|$
		and $\mu=\mu_2(\Omega)$.
		
		We consider the fundamental solution of the Helmholtz equation $\Delta \Gamma + \mu \Gamma = -\delta_p$ in $\mathbb{R}^2$,
		where $\delta_p$ is the Dirac delta function.
		It is well known that the fundamental solution is given by
		\begin{equation}
			\Gamma(x, p)=\Gamma(|x-p|) = -\frac{1}{4} Y_0(\sqrt{\mu} r),
		\end{equation}
		where $r=|x-p|$ and $Y_0$ is the Bessel function of the second kind of order zero.
		We apply Green’s identity to $u$ and $\Gamma$ on $\Omega-B_\varepsilon(p)$,
		then
		\begin{equation}
			\int_{\Omega-B_\varepsilon(p)} (u \Delta\Gamma - \Gamma \Delta u) \, dx = \int_{\partial (\Omega-B_\varepsilon(p))} \left( u \frac{\partial \Gamma}{\partial n} - \Gamma\frac{\partial u}{\partial n} \right) ds.
		\end{equation}
		Let $\varepsilon \to 0$, by using the Neumann boundary condition, together with the fact
		$\Gamma(x,p) \sim -\frac{1}{2\pi}\ln r$ as $r\to 0$,
		we have
		\begin{align}\label{eq:green-3id}
			M = -\int_{\partial \Omega} u(x) \frac{\partial \Gamma}{\partial n}(x,p) ds(x).
		\end{align}
		Using the identity $Y_0' = -Y_1$, we compute the normal derivative of $\Gamma$, then
		\begin{align}\label{eq:deri-of-Gamma}
			\frac{\partial \Gamma}{\partial n} = -\frac{1}{4} \sqrt{\mu} Y_0'(\sqrt{\mu} r) \frac{\partial r}{\partial n} = \frac{\sqrt{\mu}}{4} Y_1(\sqrt{\mu} r) \frac{\partial r}{\partial n}.
		\end{align}
		We now compute in polar coordinates centered at $p$,
		the arc length element $ds$ and the angle element $d\theta$ satisfy the fundamental relation
		\begin{align} \label{eq:polar-ds}
			\frac{\partial r}{\partial n}ds = r d\theta.
		\end{align}
		Substituting \eqref{eq:deri-of-Gamma} and \eqref{eq:polar-ds} into \eqref{eq:green-3id}, we have
		\begin{align} \label{eq:key-id}
			M = -\frac{1}{4} \int_0^{2\pi} u(x) \sqrt{\mu} r(\theta) Y_1(\sqrt{\mu} r(\theta))  d\theta.
		\end{align}
		Let  $z(\theta) = \sqrt{\mu} r(\theta)$, then \eqref{eq:key-id} becomes
		\begin{align}\label{eq:key-id2}
			M = -\frac{1}{4} \int_0^{2\pi} u(x) z(\theta)  Y_1(z(\theta))  d\theta.
		\end{align}

		Let $\Phi(x,p)=\Phi(|x-p|)=J_0(\sqrt{\mu}r)$, it satisfies
		$\Delta\Phi + \mu \Phi = 0$ in $\mathbb{R}^2$.
		Applying Green’s identity to $u$ and $\Phi$, analogously to the preceding argument and computation, we obtain
		\begin{align}\label{eq:key-id3}
			0 =  \int_0^{2\pi} u(x) z(\theta) J_1(z(\theta)) d\theta.
		\end{align}
		Multiplying \eqref{eq:key-id3} by an arbitrary constant $-c/4$ and adding it to \eqref{eq:key-id2}, we obtain
		\begin{align}
			M=-\frac{1}{4} \int_0^{2\pi} u(x) z(\theta)  \big(Y_1(z(\theta))+cJ_1(z(\theta))\big)  d\theta.
		\end{align}
		Since $|u(x)| \le m$ on the boundary, taking the supremum over all second Neumann eigenfunctions,
		we have
		\begin{equation}\label{eq:euclidean-average-bound}
			\mathfrak C(\Omega)\le\frac14\int_0^{2\pi}|F_c(z(\theta))| d\theta,  \quad\text{for any constant } c.
		\end{equation}

		It is well known that the second Neumann eigenvalue on convex domains has the following bound (\cite[Page 293]{Cheng1975MathZ})
		\begin{align}
			\mu D^2 \le 4 j_{0,1}^2,
		\end{align}
		where $D$ is the diameter of $\Omega$.
		Since $r(\theta)\le D$, then $z(\theta)=\sqrt{\mu} r(\theta)\le \sqrt{\mu}D\le 2j_{0,1}$. 
		On the other hand, the area formula gives
		\[
		|\Omega|=\frac12\int_0^{2\pi}r(\theta)^2\,d\theta.
		\]
		Consequently,
		\begin{equation}\label{eq:z-second-moment}
			\frac1{2}\int_0^{2\pi}z(\theta)^2 d\theta
			=\mu|\Omega|
			\le \pi (j'_{1,1})^2,
		\end{equation}
		where the last inequality is the sharp isoperimetric inequality $\mu |\Omega|\le\pi(j'_{1,1})^2$
		for the second Neumann eigenvalue in plane \cite{Ashbaugh1995JLMS}.

		Now let $(c,\alpha,\beta)$ be any feasible triple in
		\eqref{eq:minimax}. From \eqref{eq:euclidean-average-bound} and
		\eqref{eq:z-second-moment}, we have
		\begin{align}
			\mathfrak C(\Omega) \le\frac14\int_0^{2\pi}|F_c(z(\theta))| d\theta
			\le \frac14\int_0^{2\pi}
			\bigl(\alpha+\beta z(\theta)^2\bigr)d\theta
			\le\frac{\pi}{2}\bigl(\alpha+\beta (j'_{1,1})^2\bigr).
		\end{align}
		Taking the infimum over all feasible triples proves
		\eqref{eq:general-euclidean-bound}.
	\end{proof}

	\begin{lemma}\label{lem:euclidean-explicit-majorant}
		For $c=\frac15$, one has
		\begin{equation}\label{eq:explicit-bessel-majorant}
			\left|z\left(Y_1(z)+\frac15J_1(z)\right)\right|
			\le	\frac7{10}+\frac{z^2}{14},\qquad 0< z\le2j_{0,1}.
		\end{equation}
	\end{lemma}
	\begin{proof}
		The proof below is somewhat tedious but essentially trivial,
		it merely consists of comparing the values of two functions. See Figure \ref{fig:1}.
		Set
		\begin{align}
			F(z)=z\left(Y_1(z)+\frac15J_1(z)\right),\quad
			H(z)=\frac7{10}+\frac{z^2}{14},\quad
			Q(z)=Y_0(z)+\frac15J_0(z).
		\end{align}
		Although $Y_1$ is singular at the origin, $zY_1(z)$ still has the limit $-\frac2\pi$ as $z\to 0$.
		Thus $F$ admits a continuous extension to
		$z=0$, with
		\(
		F(0)=-\frac2\pi.
		\)

		The desired estimate \eqref{eq:explicit-bessel-majorant} is
		equivalent to
		\begin{equation}\label{eq:two-positive-functions}
			H(z)+F(z)\ge0,
			\qquad
			H(z)-F(z)\ge0.
		\end{equation}
		We therefore study the minima of the two functions $H+F$ and
		$H-F$ on the interval
		$[0,L]$, where $L=2j_{0,1}$.
		Using the recurrence relations of
		Bessel functions (\cite[(10.6.2)]{DLMF}), we have
		\begin{align}
			\frac{d}{dz}\bigl(zJ_1(z)\bigr)=& J_1(z)+z(J_0(z)-\frac{1}{z}J_1(z)) = zJ_0(z), \\
			\frac{d}{dz}\bigl(zY_1(z)\bigr)=& Y_1(z)+z(Y_0(z)-\frac{1}{z}Y_1(z)) = zY_0(z).
		\end{align}
		It follows that
		\begin{equation}\label{eq:HF-derivatives}
			(H+F)'(z)
			=
			z\left(Q(z)+\frac17\right),
			\qquad
			(H-F)'(z)
			=
			z\left(\frac17-Q(z)\right).
		\end{equation}
		Consequently, it is therefore enough to check $H+F$ at the endpoints and at the
		solutions of $Q=-1/7$, and to check $H-F$ at the endpoints and at
		the solutions of $Q=1/7$.

		We next analyze the behavior of $Q$.
		Differentiating
		$Q$ and using $J_0'=-J_1$ and $Y_0'=-Y_1$ gives
		\[
		Q'(z)
		=
		-\left(Y_1(z)+\frac15J_1(z)\right)=J_1(z)\left(R(z)-\frac15\right),
		\]
		where $R(z)=-\frac{Y_1(z)}{J_1(z)}$.
		By using the Bessel Wronskian
		identity (\cite[(10.5.2)]{DLMF}), we have 
		\begin{equation}\label{eq:R1-monotone}
			R'(z)
			=
			-\frac{2}{\pi zJ_1(z)^2}<0,
		\end{equation}
		on every interval on which $J_1$ does not vanish.
		Thus $R$ is strictly decreasing on every interval free of zeros of
		$J_1$.
		First, on the interval $(0,j_{1,1})$, one has $J_1(z)>0$. The standard Bessel expansions (\cite[(10.2.2), (10.8.1)]{DLMF})
		\[
		J_1(z)=\frac z2+O(z^3),
		\qquad
		Y_1(z)=-\frac{2}{\pi z}+O(z\log z)
		\]
		show that $R(z)\to+\infty$ as $z\to0^+$.
		On the other hand, as $z\to j_{1,1}^{-}$,
		$J_1(z)\to0^+$ and $Y_1(j_{1,1})>0$, so
		$R(z)\to-\infty$.
		Since $R$ is strictly decreasing, the equation
		$R(z)=1/5$
		has exactly one solution in $(0,j_{1,1})$. Hence $Q'$ has exactly
		one zero on $(0,j_{1,1})$.
		Second, on the interval $(j_{1,1},L)$,
		since $L<j_{1,2}$, so $J_1$ has no further zero in $(j_{1,1},L]$.
        A rigorous numerical evaluation\footnote{All decimal bounds involving special functions below were obtained by interval evaluation with outward rounding; only the resulting enclosures needed for the indicated sign estimates are displayed.}
        gives $R(L)>0.70>1/5$.
		Since $R$ is strictly decreasing on $(j_{1,1},L]$ and
		$R(z)\to+\infty$ as $z\to j_{1,1}^{+}$.
		Therefore $Q'$ has no zero on $(j_{1,1},L]$.
		In summary, $Q$ has exactly one critical point on $(0,L]$, 
		it increases up to a unique maximum and then decreases thereafter.
		In particular, each of the equations
		\[
		Q(z)=-\frac17,
		\qquad
		Q(z)=\frac17
		\]
		has at most two solutions in $(0,L)$.
		
		We now locate these solutions. A rigorous numerical evaluation
		gives $Q(0.58)+\frac17<-0.007<0$ and $Q(0.59)+\frac17>0.004>0$,
		thus there exists $a_1\in(0.58,0.59)$
		such that $Q(a_1)=-\frac17$.
		The same numerical evaluation gives $Q(4.12)+\frac17>0$ and $Q(4.13)+\frac17<0$,
		and therefore there exists $a_2\in(4.12,4.13)$
		such that $Q(a_2)=-\frac17$.
		Since the equation $Q=-\frac17$ has at most two solutions, these are
		its only solutions in $(0,L)$.
		Likewise, there exists $b_1\in(0.86,0.88)$
		and $b_2\in(3.43,3.44)$ such that they are
		the only solutions of $Q=\frac 17$ in $(0,L)$.

		It remains to check the values of $H+F$ at the endpoints
		and at their respective critical points.
		Let $G(z)=H(z)+F(z)$.
		At the left endpoint,
		$G(0)=\frac7{10}-\frac2\pi>0$,
		at the right endpoint $L=2j_{0,1}$, the same certified evaluations yield
		$G(L)>3$.
		For the critical points of $G$, namely $a_1$ and $a_2$,
		we illustrate the verification at the first critical point
		$a_1\in(0.58,0.59)$.
		For any $z\in (0.58,0.59)$,
		since $Q$ is monotonic over the interval $(0.58,0.59)$,
		then we have
		\[
		\big|G'(z)\big|=z\big|Q(z)+\frac17\big|< \max\{\big|Q(0.58)+\frac17\big|,\big|Q(0.59)+\frac17\big| \}<1\times10^{-2}.
		\]
		Since $|a_1-0.58|<10^{-2}$, the mean value theorem yields
		\[
		\begin{aligned}
			G(a_1)
			&\ge G(0.58)
			-\sup_{0.58\le z\le0.59}|G'(z)|
			\,|a_1-0.58|\\
			&>0.003-10^{-4}
			>0.002.
		\end{aligned}
		\]
		Similarly, one obtains $G(a_2)>3.39$.
		Thus $G$ is positive at both endpoints and at every interior
		critical point. We conclude that
		\begin{align}
			H(z)+F(z)\ge0,\qquad 0\le z\le L.
		\end{align}
		Similarly, applying the same argument, we have
		\begin{align}
			H(z)-F(z)\ge0,\qquad 0\le z\le L.
		\end{align}

		Combining the last two inequalities yields
		\begin{align}
			-H(z)\le F(z)\le H(z),\qquad0\le z\le L,
		\end{align}
		and therefore
		\begin{align}
			\left|z\left(Y_1(z)+\frac15J_1(z)\right)\right|
			\le\frac7{10}+\frac{z^2}{14},\qquad0< z\le2j_{0,1}.
		\end{align}
		This proves \eqref{eq:explicit-bessel-majorant}.
	\end{proof}
	
	\begin{figure}[htb]
		\centering
		\includegraphics[width=0.8\textwidth]{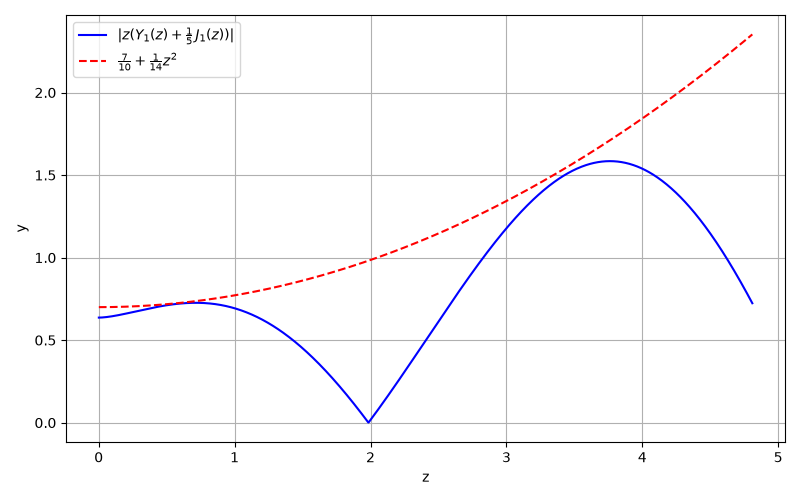}
		\caption{$\left|z\left(Y_1(z)+\frac15J_1(z)\right)\right|$ and $\frac{7}{10}+\frac{z^2}{14}$ on $[0,2j_{0,1}]$}\label{fig:1}
	\end{figure}

	Taking $c=1/5$ in Theorem \ref{thm:hot-spots-constant} and using the bound in Lemma \ref{lem:euclidean-explicit-majorant},
	we obtain a bound for the hot‑spots constant.
	\begin{corollary}\label{thm:euclidean-hot-spots-improved}
		Let $\Omega\subset\mathbb R^2$ be a bounded convex domain. Then
		\begin{equation}\label{eq:improved-euclidean-hot-spots}
			\mathfrak C(\Omega)\le\frac{\pi}{2}\left(\frac7{10}+\frac{(j'_{1,1})^2}{14}\right)< 1.48.
		\end{equation}
	\end{corollary}

	\begin{remark} \label{rem:euclidean-numerical-optimizer}
		For comparison, numerical computations suggest that the optimal parameters are approximately
		\[
		c_*\approx0.1754,
		\qquad
		\alpha_*\approx0.7021,
		\qquad
		\beta_*\approx0.0668,
		\]
		with the two contact points $x_*\approx0.6125,y_*\approx3.4741$. 
		Thus numerical computations suggest that the
		minimax problem may yield the approximate bound
		\[
		\mathfrak C(\Omega)\le\frac{\pi}{2}\mathscr M_{\rm num}
		=\frac{\pi}{2}\left(\alpha_*+\beta_*(j'_{1,1})^2\right) \approx 1.459
		\]
	\end{remark}

	\subsection{Hot spots constant in $\Sph$}
	
	Following the approach of Theorem \ref{thm:hot-spots-constant}, we derive the following result on the sphere.
	
	\begin{proposition}\label{thm:hot-spots-constant-sphere}
		Let $\Omega\subset \Sph$ be a bounded convex domain contained in a hemisphere, with diameter $D = \diam_{\Sph}(\Omega)$ and $\mu = \mu_2(\Omega)$. Let $\nu$ be the positive constant that satisfies $\nu(\nu+1) = \mu$. Then
		\begin{equation}\label{eq:hot-spots-constant-sph}
			\mathfrak{C}(\Omega) \le \inf_{c\in\R}\sup_{r\in (0, D)} \Big| \big(\mathsf{Q}_\nu'(\cos r)+c \mathsf{P}_\nu'(\cos r)\big)\Big| \sin^2 r ,
		\end{equation}
		where $\mathsf{P}_\nu$ and $\mathsf{Q}_\nu$ are Ferrers functions of the first and second kind of degree $\nu$, respectively.
	\end{proposition}
	\begin{proof}
		Without loss of generality, we assume that the global maximum of $|u|$ is attained at an interior point $p\in \Omega$.
		Replacing $u$ by $-u$ if necessary, choose $p\in\Omega$ such that $M=u(p)>0$. 
		Set $m = \max\limits_{x \in \partial\Omega} |u(x)|$.
		We consider the fundamental solution of the Helmholtz equation $\Delta \Gamma + \mu \Gamma = -\delta_p$ in $\Sph\backslash\{-p\}$.
		It is well known that the fundamental solution is given by
		\begin{equation}
			\Gamma(x, p) = \Gamma(r) = \frac{1}{2\pi} \mathsf{Q}_\nu(\cos r),
		\end{equation}
		where $r = d_{\Sph}(x, p)$.
		As $r \to 0$, $\Gamma(r) \sim -\frac{1}{2\pi}\ln r$, which yields the required Dirac delta singularity.
		
		Applying Green’s identity to $u$ and $\Gamma$ on $\Omega \backslash \mathcal{B}_\varepsilon(p)$,
		as in the Euclidean case, we have
		\begin{align}\label{eq:green-3id-sph}
			M = -\int_{\partial \Omega} u(x) \frac{\partial \Gamma}{\partial {n}}(x,p) ds(x),
		\end{align}
		and the normal derivative is
		\begin{align}\label{eq:deri-of-Gamma-sph}
			\frac{\partial \Gamma}{\partial n} = -\frac{1}{2\pi} \mathsf{Q}_\nu'(\cos r) \sin r \frac{\partial r}{\partial n}.
		\end{align}
		In geodesic polar coordinates on $\Sph$, the metric is $ds^2 = dr^2 + \sin^2 r d\theta^2$. Geometric analysis of the boundary yields the relation
		\begin{align} \label{eq:polar-ds-sph}
			\frac{\partial r}{\partial n}ds = \sin r d\theta.
		\end{align}
		Substituting \eqref{eq:deri-of-Gamma-sph} and \eqref{eq:polar-ds-sph} into \eqref{eq:green-3id-sph},
		we have
		\begin{align} \label{eq:key-id-sph}
			M = \frac{1}{2\pi} \int_0^{2\pi} u(x) \mathsf{Q}_\nu'(\cos r(\theta)) \sin^2 r(\theta) d\theta.
		\end{align}

		Let $\Phi(x,p) = \mathsf{P}_\nu(\cos r)$ be the regular solution that satisfies $\Delta\Phi + \mu \Phi = 0$ in $\Sph\backslash\{-p\}$. We have
		\begin{align}\label{eq:key-id3-sph}
			0 =  \int_0^{2\pi} u(x) \mathsf{P}_\nu'(\cos r(\theta)) \sin^2 r(\theta) d\theta.
		\end{align}
		Multiplying \eqref{eq:key-id3-sph} by an arbitrary constant $\frac{c}{2\pi}$ and adding it to \eqref{eq:key-id-sph}, we obtain
		\begin{align}
			M = \frac{1}{2\pi} \int_0^{2\pi} u(x) \big(\mathsf{Q}_\nu'(\cos r)+c \mathsf{P}_\nu'(\cos r)\big) \sin^2 r d\theta.
		\end{align}

		Finally, we derive an upper bound. Since $|u(x)| \le m$ on the boundary, we have
		\begin{equation}
			\begin{aligned}
				M &\le \frac{m}{2\pi} \int_0^{2\pi}  \Big| \big(\mathsf{Q}_\nu'(\cos r)+c \mathsf{P}_\nu'(\cos r)\big)\Big| \sin^2 r  d\theta \\
				&\le m \max_{\theta } \Big| \big(\mathsf{Q}_\nu'(\cos r(\theta))+c \mathsf{P}_\nu'(\cos r(\theta))\big)\Big| \sin^2 r(\theta) .
			\end{aligned}
		\end{equation}
		Since the domain is bounded by diameter $D$, $r(\theta) < D$. Taking the supremum over $r \in (0, D)$ and infimum over $c$, this completes the proof.
	\end{proof}

	In Proposition \ref{thm:hot-spots-constant-sphere},
	the upper bound on the right side
	of \eqref{eq:hot-spots-constant-sph} still depends on $\mu$ and $D$.
	Compared to Theorem~\ref{thm:hot-spots-constant}, 
	the main difficulty is that the variable $\mu$ and the radial distance $r$ are coupled into a single variable $z = \sqrt{\mu}r$ in the Euclidean case. 
	However, $\mu = \nu(\nu+1)$ and $r$ appear separately in the spherical case. 
	This difference from the Euclidean setting forces us to adopt a different procedure in what follows.

	To estimate the perturbation of the Bessel equation arising below, we first record the following Green representation formula for the corresponding nonhomogeneous equation.
	
	\begin{lemma}\label{lem:bessel-apro}
		Let $f(t)=O(t^\alpha)$ for some $\alpha>0$, and $G$ solve
		\begin{equation}\label{eq:g-model-eq}
			G''- \frac{1}{t}G' +G=f(t),\quad 0<t<T.
		\end{equation}
		Let $U(t)=tJ_1(t)$ and $V(t)=-\frac{\pi}{2}tY_1(t)$ be the two fundamental solutions
		of the corresponding homogeneous equation.
		Suppose that 
		$H(t)=aV(t)+bU(t) $
		is chosen so that $G(t)-H(t)=o(t^2)$ as $t\to0^+$.
		Then
		\begin{equation}\label{eq:general-green-representation}
			G(t)=H(t)+\int_0^t K(t,s)f(s)ds,
		\end{equation}
		where $K(t,s)=\big(U(t)V(s)-V(t)U(s)\big)s^{-1}$ is the corresponding Green kernel.
	\end{lemma}
	
	\begin{proof}
		Let $R(t)=G(t)-H(t)$.
		Then $R$ solves the non-homogeneous equation,
		\begin{align} \label{eq:non-homo}
			R''-\frac1tR'+R=f(t).
		\end{align}
		It is well known that the general solution of \eqref{eq:non-homo} is as follows:
		\begin{align}\label{eq:non-homo-gen-solution}
			R(t)=c_1U(t)+c_2V(t)+\int_0^tK(t,s)f(s)ds.
		\end{align}
		The standard Bessel expansions (\cite[(10.2.2), (10.8.1)]{DLMF})
		\begin{align}
			U(t)=\frac{t^2}{2}+O(t^4),\quad V(t)=1+O(t^2\log t)
		\end{align}
		imply, uniformly for sufficiently small $0<s<t$,
		\begin{align}
			K(t,s)\le C\frac{t^2}{s}.
		\end{align}
		Therefore, the integral term in \eqref{eq:non-homo-gen-solution} is $o(t^2)$ as $t\to0^+$.
		Hence the matching condition $R(t)=o(t^2)$
		forces $c_1=c_2=0$.
		Therefore,
		\[ G(t)=H(t)+\int_0^t K(t,s)f(s) ds.\]
		This proves the lemma.
	\end{proof}
	
	We next establish uniform bounds for the fundamental solutions of the Bessel-type equation and their derivatives on the interval needed in the subsequent estimates.
	
	\begin{lemma}\label{lem:U-V-estimate}
		Let $U(t)=tJ_1(t)$ and $V(t)=-\frac{\pi}{2}tY_1(t)$.
		Then, for every $0<t\le 5$, we have
		\begin{equation}\label{eq:global-U-V-estimate}
			|U(t)|\le \frac{19}{10},\quad|U'(t)|\le \frac{19}{10},\quad|V(t)|\le \frac{27}{10},\quad|V'(t)|\le \frac{27}{10}.
		\end{equation}
	\end{lemma}
	
	\begin{proof}
		Clearly, both $U$ and $V$ solve the homogeneous equation $F''(t)-\frac{1}{t}F'(t)+F(t)=0$.
		For any solution $F$, let the  energy $\mathcal E (t)=F(t)^2+F'(t)^2$, then we have
		\begin{align}
			\mathcal E '(t)=2F'(t)\bigl(F(t)+F''(t)\bigr)=\frac{2}{t}F'(t)^2\ge0.
			\label{eq:UV-energy-monotone}
		\end{align}
		Thus $\mathcal E(t)$ is nondecreasing on $(0,\infty)$.
		It follows from \eqref{eq:UV-energy-monotone} that, for
		$0<t\le 5$,
		\begin{equation}\left\{ 
			\begin{aligned}\label{eq:U-V-estimate}
				|U(t)|, |U'(t)|\le \sqrt{U(5)^2+U'(5)^2}, \\
				|V(t)|, |V'(t)|\le \sqrt{V(5)^2+V'(5)^2}.
			\end{aligned}\right.
		\end{equation}

		Using the standard identities
		$\bigl(tJ_1(t)\bigr)'=tJ_0(t)$ and
		$\bigl(tY_1(t)\bigr)'=tY_0(t)$,
		we have $U'(t)=tJ_0(t)$ and $V'(t)=-\frac{\pi}{2}tY_0(t)$.
		A rigorous numerical evaluation gives
		\[
		-1.639<U(5)<-1.637,\qquad -0.889<U'(5)<-0.887,
		\]
		\[
		-1.162<V(5)<-1.160,\qquad 2.422<V'(5)<2.424.
		\]
		Consequently,
		$U(5)^2+U'(5)^2<1.639^2+0.889^2<1.9^2$ and
		$V(5)^2+V'(5)^2<1.162^2+2.424^2<2.7^2$.
		Then \eqref{eq:U-V-estimate} yields
		\begin{align}
			|U(t)|\le \frac{19}{10},\quad|U'(t)|\le \frac{19}{10},\quad|V(t)|\le \frac{27}{10},\quad|V'(t)|\le \frac{27}{10}.
		\end{align}
		This proves the lemma.
	\end{proof}
	
	To identify the appropriate homogeneous Bessel approximation of $G_\nu$ near the origin,
	we next determine its precise asymptotic expansion as $t\to0^+$.
	
	\begin{lemma}\label{lem:G-nu-asymptotic}
		Let $\nu>0$, $\mu=\nu(\nu+1)$, and 
		\begin{align}
			g(r,\nu)=\sin^2 r \mathsf{Q}_\nu'(\cos r),\quad G_\nu(t)=g\left(\frac{t}{\sqrt{\mu}},\nu\right).
		\end{align}
		Then, as $t\to0^+$,
		\begin{align}\label{eq:G-nu-asymptotic}
			G_\nu(t)=&1-\frac{t^2}{2}\ln t+t^2\left(\frac14\ln\mu+\frac14+\frac12\ln2-\frac{\gamma}{2}-\frac12\psi(\nu+1)\right)
			+O(t^4\ln t),
		\end{align}
		where $\gamma\approx 0.5772$ denotes the Euler-Mascheroni constant and
		$\psi(x)$ is  the logarithmic derivative of the Gamma function, also known as the digamma function (or psi function).
	\end{lemma}
	
	\begin{proof}
		Let $F_\nu(r)=\mathsf{Q}_\nu(\cos r)$. 
		Since $\mathsf{Q}_\nu$ satisfies the Legendre equation, $F_\nu$ satisfies
		\begin{equation}\label{eq:F-nu-radial-equation}
			F_\nu''(r)+\cot r\,F_\nu'(r)+\mu F_\nu(r)=0.
		\end{equation}
		The standard expansion for $\mathsf{Q}_\nu$ as $x\to 1^-$ (\cite[(14.8.3)]{DLMF}) reads
		\begin{align}
			\mathsf{Q}_\nu(x)=\frac12\ln\left(\frac{2}{1-x}\right)-\gamma -\psi(\nu+1)+O\big((x-1)\ln(1-x)\big),
		\end{align}
		together with $1-\cos r=\frac{r^2}{2}+O(r^4)$, gives
		\begin{equation}\label{eq:F-nu-leading-asymptotic}
			F_\nu(r)=-\ln r+C_\nu+O(r^2\ln r),\qquad
			C_\nu=\ln 2-\gamma-\psi(\nu+1).
		\end{equation}
		Multiplying \eqref{eq:F-nu-radial-equation} by $\sin r$ gives
		\begin{equation}\label{eq:F-nu-divergence}
			\bigl(\sin r F_\nu'(r)\bigr)'=-\mu\sin r F_\nu(r).
		\end{equation}
		By \eqref{eq:F-nu-leading-asymptotic}, the right-hand side in \eqref{eq:F-nu-divergence} is
		$O(r|\ln r|)$ near $r=0$. Hence
		$\sin r F_\nu'(r)$ has a finite limit as $r\to0^+$.
		Let this limit be $\ell$. Integrating \eqref{eq:F-nu-divergence} near zero yields
		\begin{align}
			\sin r F_\nu'(r)= \ell+O(r^2|\ln r|),
		\end{align}
		so $F_\nu'(r)=\frac{\ell}{r}+O(r|\ln r|)$, and therefore $F_\nu(r)=\ell\ln r+O(1)$.
		Comparing this with \eqref{eq:F-nu-leading-asymptotic} gives
		$\ell=-1$. Thus, by using the identity $g(r,\nu)=-\sin r F_\nu'(r)$, we have
		\begin{align}
			\lim_{r\to0^+}g(r,\nu)= \lim_{r\to0^+} -\sin r F_\nu'(r)=1.
		\end{align}
		Moreover, \eqref{eq:F-nu-divergence} implies $g'(r,\nu)=\mu\sin r\,F_\nu(r)$, 
		and hence
		\begin{equation}\label{eq:g-integral-representation}
			g(r,\nu)
			=
			1+\mu\int_0^r\sin s F_\nu(s) ds.
		\end{equation}
		Using $\sin s=s+O(s^3)$,  together with \eqref{eq:F-nu-leading-asymptotic}, we have
		\begin{align}
			\int_0^r\sin s F_\nu(s)ds=-\frac{r^2}{2}\ln r+\left(\frac14+\frac{C_\nu}{2}\right)r^2+O(r^4|\ln r|).
		\end{align}
		Substitution into \eqref{eq:g-integral-representation} gives
		\begin{equation}\label{eq:g-asym-exp}
			g(r,\nu)=1-\frac{\mu}{2}r^2\ln r+\mu r^2\left(\frac14+\frac12C_\nu\right)+O(r^4|\ln r|).
		\end{equation}
		Finally, substituting $r=t/\sqrt{\mu}$ into \eqref{eq:g-asym-exp}, we obtain
		\begin{align*}
			G_\nu(t)=1-\frac{t^2}{2}\ln t+
			t^2\left(\frac14\ln\mu+\frac14+\frac12\ln2-\frac{\gamma}{2}-\frac12\psi(\nu+1)\right)
			+O(t^4|\ln t|),
		\end{align*}
		This proves the lemma.
	\end{proof}

	Combining the preceding uniform estimates, and the Green representation formula,
	we now obtain an explicit universal upper bound for the spherical case.
	
	\begin{theorem}\label{thm:sphere-universal-const-4}
		Let $\Omega\subset \Sph$ be a bounded convex domain contained in a hemisphere, then
		\begin{equation}\label{eq:sphere-explicit-universal-constant}
			\mathfrak C(\Omega) < 4.
		\end{equation}
	\end{theorem}
	
	\begin{proof}
		Let $D = \diam_{\Sph}(\Omega)$, $\mu = \mu_2(\Omega)$ and $\nu>0$ be determined by $\nu(\nu+1)=\mu$.
		If $\mu\le 2$, then {\color{blue}Remark \ref{rem:spherical-low-mu-recovery}}
		implies that the second Neumann eigenfunction has no interior critical
		point, and $\mathfrak{C}(\Omega)=1$ is trivial. Therefore we may assume that $\mu\ge2$.
		We divide the proof into two ranges of the degree $\nu$.

		\noindent
		\textbf{Case 1: $1\le\nu\le 8$.}

		For $c\in\mathbb R$, let $R_c(x)=\mathsf Q_\nu(x)+c\mathsf P_\nu(x)$, and $h_c(r)=\sin^2r\,R_c'(\cos r)$.
		By Proposition~\ref{thm:hot-spots-constant-sphere},
		\begin{equation}\label{eq:sphere-low-by-hc}
			\mathfrak C(\Omega)\le \inf_{c\in\mathbb R}\sup_{0<r<D}|h_c(r)|.
		\end{equation}
		Using the Legendre differential equation, a direct computation shows that $h_c(r)$ satisfies
		\begin{equation}\label{eq:sphere-hc-ode}
			h_c''-\cot r h_c'+\mu h_c=0.
		\end{equation}
		Let the energy $\mathcal E_c(r)=h_c(r)^2+\frac{h_c'(r)^2}{\mu}$, then we have
		\begin{equation}\label{eq:sphere-energy-derivative}
			\mathcal E_c'(r)=\frac{2\cot r}{\mu} h_c'(r)^2.
		\end{equation}
		Thus $\mathcal E_c$ is increasing on $(0,\pi/2)$ and decreasing on
		$(\pi/2,\pi)$, therefore we obtain
		\begin{equation}\label{eq:sphere-energy-midpoint}
			\sup_{0<r<D}|h_c(r)|^2\le \mathcal E_c \left(\frac\pi2\right).
		\end{equation}
		At $r_0=\pi/2$, it is clear that $h_c(r_0)=R_c'(0)$.
		Using the fact that $R_c$ satisfies the Legendre equation, a straightforward computation gives
		$h_c'(r_0)=\mu R_c(0)$. Therefore
		\begin{align}
			\mathcal E_c\left(\frac\pi2\right)
			=\bigl(\mathsf Q_\nu'(0)+c\mathsf P_\nu'(0)\bigr)^2
			+\mu\bigl(\mathsf Q_\nu(0)+c\mathsf P_\nu(0)\bigr)^2.
		\end{align}
		Minimizing this quadratic expression in $c$, we have
		\begin{equation}\label{eq:sphere-energy-minimum}
			\inf_{c\in\mathbb R}\mathcal E_c \left(\frac\pi2\right)
			=\frac{\mu \big(\mathsf P_\nu(0)\mathsf Q_\nu'(0)-\mathsf P_\nu'(0)\mathsf Q_\nu(0) \big)^2 }{\mathsf P'_\nu(0)^2+\mu \mathsf P_\nu(0)^2}
			=\frac{\mu}{\mathsf P_\nu'(0)^2+\mu\mathsf P_\nu(0)^2}.
		\end{equation}
		where the last equality follows from the Legendre Wronskian identity
		$\mathsf P_\nu(0)\mathsf Q_\nu'(0)-\mathsf P_\nu'(0)\mathsf Q_\nu(0)=1$ (\cite[(14.2.4)]{DLMF}).
		Using $\mathsf P_\nu'(0)=\nu\mathsf P_{\nu-1}(0)$ (\cite[14.10.5]{DLMF}), we obtain
		\begin{equation}\label{eq:sphere-low-degree-reduction}
			\mathfrak C(\Omega)^2 
			\le
			\frac{\nu+1}{\nu\mathsf P_{\nu-1}(0)^2+(\nu+1)\mathsf P_\nu(0)^2}.
		\end{equation}

		We now use the standard value (\cite[14.5.1]{DLMF})
		\begin{align}  \label{eq:P-0-exp}
			\mathsf P_\nu(0)
			=\frac{\sqrt\pi}{\Gamma\left(\frac{1}{2}-\frac{\nu}{2}\right)\Gamma\left(1+\frac\nu2\right)}
			=\frac{\sqrt\pi}{\Gamma\left(\frac{1}{2}-\xi\right)\Gamma\left(1+\xi \right)}
			=\frac{\sqrt\pi}{\Gamma\left(\frac{1}{2}-\xi\right)\Gamma\left(\xi \right)\xi}.
		\end{align}
		where $\xi=\frac \nu 2$.
		The Gamma reflection formula $\Gamma(x)\Gamma(1-x)=\frac{\pi}{\sin(\pi x)}$ (\cite[5.5.3]{DLMF})
		gives $\Gamma(\frac 12 -\xi )\Gamma(\frac 12 +\xi )=\frac{\pi }{\cos(\pi\xi)}$,
		substituting this into \eqref{eq:P-0-exp} gives
		\begin{align}
			\mathsf P_\nu(0)=\frac{R(\xi)}{\sqrt\pi \xi}\cos(\frac{\pi\nu}{2}),
		\end{align}
		where $R(\xi)=\frac{\Gamma(\xi+\frac12)}{\Gamma(\xi)}$.
		Similarly, we have $$\mathsf P_{\nu-1}(0)=\frac{1}{\sqrt\pi R(\xi)}\sin(\frac{\pi\nu}{2}).$$
		By using Gautschi's inequality (\cite[(5.6.4)]{DLMF}) with $s=1/2$ and $\Gamma(1+\xi)=\xi\Gamma(\xi)$, we have 
		$\xi(\xi+1)^{-\frac 12}<R(\xi)<\sqrt{\xi}$.
		It follows that
		\begin{align}
			\nu\mathsf P_{\nu-1}(0)^2+(\nu+1)\mathsf P_\nu(0)^2 &=\frac{\nu}{R^2(\xi)}\frac{1}{\pi}\sin^2(\frac{\pi\nu}{2})
			+(\nu+1)\frac{R^2(\xi)}{\xi^2}\frac{1}{\pi}\cos^2(\frac{\pi\nu}{2}) \\
			&> \frac{2}{\pi} \sin^2(\frac{\pi\nu}{2})+(\nu+1)\frac{1}{\xi+1}\frac{1}{\pi}\cos^2(\frac{\pi\nu}{2}) \\
			&> 
			\frac{2(\nu+1)}{\pi(\nu+2)}.
		\end{align}
		Combining this with \eqref{eq:sphere-low-degree-reduction}, we have
		\begin{equation}\label{eq:sphere-low-degree-final}
			\mathfrak C(\Omega)^2<\frac\pi2(\nu+2)
			\le5\pi,
		\end{equation}
		and hence
		\[
		\mathfrak C(\Omega)<\sqrt{5\pi}<4.
		\]

		\noindent
		\textbf{Case 2: $\nu\ge 8$.}

        By Proposition \ref{thm:hot-spots-constant-sphere}, taking $c=0$ gives
		\begin{equation}\label{eq:sphere-C-by-g}
			\mathfrak C(\Omega)
			\le
			\sup_{0<r<D}
			\left|\mathsf{Q}_\nu'(\cos r)\right|\sin^2 r .
		\end{equation}
		Let $t=\sqrt{\mu}r$, $g(r, \nu)=\sin^2 r \mathsf{Q}_\nu'(\cos r)$ and
		$G_\nu(t)= g\left(\frac{t}{\sqrt{\mu}},\nu\right)$,
		Lemma \ref{lem:Sph-mu-estimate} yields  $0<t\le 2j_{0,1}<5$.
		Using the Legendre differential equation, a direct computation shows that $G_\nu(t)$ satisfies
		\begin{equation}\label{eq:G-scaled-equation}
			G_\nu''(t)
			-
			\frac{1}{\sqrt{\mu}}
			\cot\left(\frac{t}{\sqrt{\mu}}\right)
			G_\nu'(t)
			+
			G_\nu(t)
			=0.
		\end{equation}
		Equivalently,
		\begin{equation}\label{eq:G-bessel-perturbation}
			G_\nu''-\frac1tG_\nu'+G_\nu
			=
			-\delta_\mu(t)G_\nu',
		\end{equation}
		where
		\begin{equation}\label{eq:def-delta-mu}
			\delta_\mu(t)
			=\frac1t-\frac1{\sqrt{\mu}}\cot\left(\frac{t}{\sqrt{\mu}}\right).
		\end{equation}
		For $0<t\le 2j_{0,1}$, using the elementary inequalities
		$\frac1x-\frac{x}{2}\le\cot x\le\frac1x$ when $x<2$ , we have
		\begin{equation}\label{eq:delta-bound}
			0\le\delta_\mu(t)\le\frac{t}{2\mu}.
		\end{equation}

		Now we consider the two function $U(t)=tJ_1(t)$ and $V(t)=-\frac{\pi}{2}tY_1(t)$, 
		they are two solutions of the homogeneous equation $H''-\frac1t H'+H=0$.
		By using the standard asymptotic expansions for $J_1$ at the origin (\cite[(10.2.2)]{DLMF}), we have
		\[
		U(t)=tJ_1(t)=\frac{t^2}{2}+O(t^4),
		\]
		and the asymptotic expansions for $Y_1$ at the origin (\cite[(10.8.1)]{DLMF}) give
		\[ V(t)=-\frac{\pi}{2}tY_1(t)=1-\frac{t^2}{2}\ln t+t^2\left(\frac12\ln2-\frac{\gamma}{2}+\frac14\right)+O(t^4\ln t). \]
		We choose $c_\nu=\frac12\ln\mu-\psi(\nu+1)$,
		then Lemma \ref{lem:G-nu-asymptotic} implies that $H_\nu(t)=V(t)+c_\nu U(t)$
		has the same $t^2$-order behavior as $G_\nu(t)$.
		By the logarithmic bounds for the digamma function
		$\ln x\le \psi(x+1) \le \ln(x+1)$ for $x>0$(\cite[(1.3)]{Guo14Ana}), we have
		\begin{align}\label{eq:v-nu-estimate}
			|c_\nu|<\frac12\ln\left(1+\frac1\nu\right)<\frac1{2\nu}\le \frac{1}{16}.
		\end{align}

		Since  $H_\nu(t)$ has the same $t^2$-order behavior as $G_\nu(t)$, \eqref{eq:delta-bound} and \eqref{eq:G-nu-asymptotic}
		imply that $\delta_\mu(s)G_\nu'(s)=O(s^\alpha)$ for every $0<\alpha<2$. Then
		Lemma \ref{lem:bessel-apro} gives
		\begin{equation}\label{eq:G-variation}
			G_\nu(t)=H_\nu(t)-\int_0^t K(t,s)\delta_\mu(s)G_\nu'(s)\,ds,
		\end{equation}
		where $K(t,s)=\big(U(t)V(s)-V(t)U(s)\big)s^{-1}$ is the corresponding Green kernel. 
		Hence, Lemma \ref{lem:U-V-estimate} gives
		\begin{align}
			|K(t,s)|&\le\frac{|U(t)||V(s)|+|V(t)||U(s)|}{s}\le\frac{2(19/10)(27/10)}{s}= \frac{513}{50s}, \label{eq:K-est}\\
			|\partial_tK(t,s)|&\le\frac{|U'(t)||V(s)|+|V'(t)||U(s)|}{s}\le\frac{2(19/10)(27/10)}{s}= \frac{513}{50s}. \label{eq:K-t-est}
		\end{align}
		Lemma \ref{lem:U-V-estimate} and \eqref{eq:v-nu-estimate} give
		\begin{align}
			|H_\nu(t)|&\le|V(t)|+|c_\nu||U(t)|\le\frac{27}{10}+\frac1{16}\frac{19}{10}=\frac{451}{160},\label{eq:H-est}\\
			|H_\nu'(t)|&\le|V'(t)|+|c_\nu||U'(t)|\le\frac{27}{10}+\frac1{16}\frac{19}{10}=\frac{451}{160}. \label{eq:H-nu-est}
		\end{align}
		Then
		differentiating \eqref{eq:G-variation} and using \eqref{eq:H-nu-est}, \eqref{eq:K-t-est}, \eqref{eq:delta-bound}, yields
		\[
		|G_\nu'(t)|\le \frac{451}{160} +\frac{513}{100\mu}\int_0^t|G_\nu'(s)| ds.
		\]
		Thus, by Gronwall's inequality, we have
		\begin{equation}\label{eq:G-Gronwall}
			\begin{aligned}
				|G_\nu'(t)|\le\frac{451}{160}
				\exp\!\left(\frac{513t}{100\mu}\right),
				\qquad 0\le t\le2j_{0,1}.
			\end{aligned}
		\end{equation}
		Finally, substituting \eqref{eq:G-Gronwall} into \eqref{eq:G-variation}, 
		together with \eqref{eq:H-est}, \eqref{eq:K-est} and \eqref{eq:delta-bound}, we have
		\begin{align}
			|G_\nu(t)|
			&\le\frac{451}{160}+\frac{513}{100\mu}\int_0^t|G_\nu'(s)| ds \\
			&\le\frac{451}{160}+\frac{513}{100\mu}\int_0^t\frac{451}{160}\exp\left(\frac{513s}{100\mu}\right) ds
			=\frac{451}{160}\exp\left(\frac{513t}{100\mu}\right).
		\end{align}
		Since $\mu\ge 72$ and $t\le2j_{0,1}$, we conclude that
		\begin{align}
			|G_\nu(t)|\le\frac{451}{160}
			\exp\left(\frac{1026j_{0,1}}{7200}\right)<3.98 < 4.
		\end{align}
		Taking $c=0$ in Proposition \ref{thm:hot-spots-constant-sphere} now gives
		$\mathfrak C(\Omega)<4$.

		Combining the two cases, we conclude that
		\[	\mathfrak C(\Omega) < 4.\]
		This completes the proof.
	\end{proof}
	
	\begin{remark}\label{rem:sphere-numerical}
		To illustrate the possible magnitude of the optimal bound, we numerically evaluate
		\[ g(r,\nu)=\sin^2 r \mathsf{Q}_\nu'(\cos r) \]
		over the nontrivial parameter regime 
		$\mathcal{S}=\{(\nu,r,D)|   0\le r\le D\le\pi, \mu\ge 2,\mu D^2 \ge j_{1,1}^2 
		\text{ and }  \mu D^2+D^2/3\le 4j_{0,1}^2\}$. For $\nu\in[1,10]$,
		the numerical search yields
		$$\sup_{\mathcal{S}\cap \{\nu\le 10\}} |g(r,\nu)|\approx 2.47,$$
		with the maximum attained near $\nu\approx 10$ and $r\approx 0.376$. For $\nu\in[1,40]$,
		the numerical search yields the maximum approximately $2.5$. 
		Further numerical computations for larger values of $\nu$ suggest that this supremum approaches approximately $2.502$.

		In fact, as $\nu\to \infty$, then $D\to 0$, and therefore $\frac{r}{\sin r}\sim 1$.
		By using the asymptotic formula (\cite[(14.15.12)]{DLMF})
		\begin{align}
			\mathsf{Q}_\nu(\cos r)\sim -\frac{\pi}{2} \sqrt{\frac{r}{\sin r}}Y_0\big((\nu+\frac{1}{2})r\big)
			\sim -\frac{\pi}{2}Y_0\big((\nu+\frac{1}{2})r\big)  \quad \text{ as }\nu\to+\infty.
		\end{align}
		For large $\nu$, we have
		\begin{align}
			g(r,\nu) = -\sin(r)\frac{d}{dr}\mathsf{Q}_\nu(\cos r) 
			\sim \frac{\pi}{2} (\sin r) Y_1\big((\nu+\frac{1}{2})r\big)(\nu+\frac{1}{2}) 
			\sim \frac{\pi}{2} Y_1(z)z.
		\end{align}	
		where $z=(\nu+\frac{1}{2})r \approx \sqrt{\mu}r \le 2j_{0,1}$ by Lemma \ref{lem:Sph-mu-estimate}.
		A straightforward analysis of the extrema of $Y_1(z)z$ on $[0,2j_{0,1}]$
		shows that the maximum value of $ \left|\frac{\pi}{2}Y_1(z)z\right|$ is $-J_0(y_{0,2})^{-1}\approx 2.502$.
		Thus these numerical and asymptotic observations suggest that the bound $4$ in
		Theorem \ref{thm:sphere-universal-const-4} can likely be improved further.
		
	\end{remark}

	\subsection{Hot spots constant in $\Hyp$}
	
	Before proceeding, we establish strictly real-valued radial solutions. Let $\nu$ be the parameter satisfying $\nu(\nu+1)=-\mu$. 
	If $\mu \le \frac{1}{4}$, then $\nu$ is a real constant,
	and both Legendre functions ${P}_\nu(x)$ and ${Q}_\nu(x)$ are real-valued on $(1,\infty)$.
	However, if $\mu > \frac{1}{4}$, then $\nu$ becomes complex, given by $\nu = -\frac{1}{2} + i\sqrt{\mu - \frac{1}{4}}$.
	In this case, ${P}_\nu(x)$ remains real-valued for $x \in (1,\infty)$,
	but ${Q}_\nu(x)$ is generally complex-valued for $x \in (1,\infty)$.
	To resolve this, we define the real-valued functions
	\begin{equation}\label{eq:real-conical-solutions}
		\mathcal{P}_\nu(x) := {P}_\nu(x), \qquad \mathcal{Q}_\nu(x) := \operatorname{Re} \{ {Q}_\nu(x) \}.
	\end{equation}
	
	By this definition, both $\mathcal{P}_\nu(x)$ and $\mathcal{Q}_\nu(x)$ are strictly real-valued on $(1,\infty)$.
	Moreover, substituting $x = \cosh r$ where $r = \dist_{\Hyp}(\cdot, p)$,
	the function $\Gamma = \frac{1}{2\pi}\mathcal{Q}_\nu(\cosh r)$ satisfies
	the singular equation $\Delta\Gamma + \mu\Gamma = -\delta_p$ in $\Hyp$,
	while $\mathcal{P}_\nu(\cosh r)$ serves as the regular solution satisfying $\Delta\Phi + \mu\Phi = 0$.
	We shall exclusively use these real-valued solutions throughout this subsection.

	\begin{proposition}\label{thm:hot-spots-constant-hyp}
		Let $\Omega\subset\Hyp$ be a bounded convex domain with diameter $D=\diam_{\Hyp}(\Omega)$ and $\mu=\mu_2(\Omega)$. Let $\nu$ be the complex constant that satisfies $\nu(\nu+1)=-\mu$. Then
		\begin{equation}\label{eq:hot-spots-constant-hyp}
			\mathfrak{C}(\Omega)\le\inf_{c\in\R}\sup_{r\in(0,D)}\Bigl|\bigl(\mathcal Q_\nu'(\cosh r)+c\mathcal P_\nu'(\cosh r)\bigr)\Bigr|\sinh^2r.
		\end{equation}
	\end{proposition}
	
	\begin{proof}
		Without loss of generality, we assume that the global maximum of $|u|$ is attained at an interior point $p\in\Omega$.
		Replacing $u$ by $-u$ if necessary, choose $p\in\Omega$ such that $M=u(p)>0$. 
		Set $m=\max\limits_{x\in\partial\Omega}|u(x)|$.
		We consider the fundamental solution of the Helmholtz equation $\Delta\Gamma+\mu\Gamma=-\delta_p$ in $\Hyp$,
		which is given by
		\begin{equation}
			\Gamma(x,p)=\Gamma(r)=\frac{1}{2\pi}\mathcal Q_\nu(\cosh r),
		\end{equation}
		where $r=d_{\Hyp}(x,p)$.
		As in the spherical case, we have
		\begin{align}\label{eq:hyp-green-identy}
			M=-\frac{1}{2\pi}\int_0^{2\pi}u(x)\bigl(\mathcal Q_\nu'(\cosh r)+c \mathcal P_\nu'(\cosh r)\bigr)\sinh^2r\,d\theta,
		\end{align}
		for any constant $c$. Hence
		\begin{equation}
			\begin{aligned}
				M\le m\max_{\theta}\Bigl|\bigl(\mathcal Q_\nu'(\cosh r(\theta))+c \mathcal P_\nu'(\cosh r(\theta))\bigr)\Bigr|\sinh^2r(\theta).
			\end{aligned}
		\end{equation}
		Since the domain is bounded by diameter $D$, $r(\theta)<D$. Taking the supremum over $r\in(0,D)$ and infimum over $c$, this completes the proof.
	\end{proof}

	The following lemma provides an exponential growth estimate for $|\mathcal Q_\nu'(\cosh z)|\sinh^2 z$.
	\begin{lemma}\label{lem:hyp-uniform-conical-kernel}
		Let $\nu=-\frac12+i\rho$ with $\rho\ge0$ and $\mu=\rho^2+\frac14$.
		Then, for every $r>0$, we have
		\begin{equation}\label{eq:hyp-explicit-conical-kernel}
			\left|\mathcal Q_\nu'(\cosh r)\right|\sinh^2 r\le 2\sqrt{\mu}e^{r/2}.
		\end{equation}
	\end{lemma}
	\begin{proof}
		The integral representation formula for the Olver function $\mathbf{Q}_\nu$\cite[(14.25.2)]{DLMF},
		together with the normalization relation between the Olver and Legendre functions \cite[(14.3.10)]{DLMF},
		yields the following representation for the standard $Q_\nu$:
		\begin{equation}\label{eq:hyp-Q-integral-representation}
			Q_\nu(\cosh r)
			=\int_0^\infty\bigl(\cosh r+\sinh r\cosh s\bigr)^{-\nu-1}ds .
		\end{equation}
		For convenience, let $A(r,s)=\cosh r+\sinh r\cosh s$,
		and $B(r,s)=\sinh r+\cosh r\cosh s$.
		Since $\partial_r A(r,s)=B(r,s)$, 
		differentiating \eqref{eq:hyp-Q-integral-representation} with respect
		to $r$ gives
		\begin{equation}\label{eq:hyp-Q-integral-derivative}
			\sinh r Q_\nu'(\cosh r)
			=-(\nu+1)\int_0^\infty B(r,s)A(r,s)^{-\nu-2}ds .
		\end{equation}
		Clearly $A(r,s)>0$ and $B(r,s)>0$ for $r>0$ and $s\ge0$, moreover $\operatorname{Re}(\nu+2)=\frac32$
		and $|\nu+1|=\sqrt{\mu}$. Taking the modulus of \eqref{eq:hyp-Q-integral-derivative}, we obtain
		\begin{align}\label{eq:hyp-Q-reduction}
			\left| Q_\nu'(\cosh r)\right|\sinh^2r
			\le\sqrt{\mu}\sinh r\int_0^\infty\frac{B(r,s)}{A(r,s)^{3/2}}ds.
		\end{align}
		
		We now estimate the integral on the right-hand side explicitly.
		Let $a=1-e^{-2r}$, then $0<a<1$ and $\sinh r=\frac{a}{2}e^r$.
		Now we change variable by $y=\sinh\frac{s}{2}$, then
		\[
		\cosh s=1+2y^2,\qquad ds=\frac{2\,dy}{\sqrt{1+y^2}}.
		\]
		Moreover,
		\begin{align}\label{eq:hyp-A-transform}
			A(r,s) &= \cosh r+\sinh r(1+2y^2) = e^r\bigl(1+ay^2\bigr), \\
			B(r,s)&=\sinh r+\cosh r(1+2y^2)= e^r\bigl(1+(2-a)y^2\bigr).
		\end{align}
		Consequently,
		\begin{align}\label{eq:hyp-kernel-first-transform}
			\sinh r\int_0^\infty\frac{B(r,s)}{A(r,s)^{3/2}} ds
			=e^{r/2}a\int_0^\infty\frac{1+(2-a)y^2}
			{(1+ay^2)^{3/2}\sqrt{1+y^2}}dy .
		\end{align}
		Next let $x=\sqrt a y$,
		then \eqref{eq:hyp-kernel-first-transform} becomes
		\begin{align}\label{eq:hyp-kernel-second-transform}
			\sinh r\int_0^\infty\frac{B(r,s)}{A(r,s)^{3/2}} ds
			=e^{r/2}\int_0^\infty\frac{a+(2-a)x^2}{(1+x^2)^{3/2}\sqrt{a+x^2}}dx =:e^{r/2} I(a) .
		\end{align}
		A direct computation of the derivative of $I(a)$ yields
		\begin{align} \label{eq:I-a-prime}
			I'(a)=\int_0^\infty\frac{a(1-x^2)-2x^4}{2(1+x^2)^{3/2}(a+x^2)^{3/2}} dx,
		\end{align}
		Note that $a(1-x^2)-2x^4=2(a-x^4)-a(1+x^2)$, substituting into \eqref{eq:I-a-prime} gives
		\begin{align}
			I'(a)&=\int_0^\infty\frac{a-x^4}{(1+x^2)^{3/2}(a+x^2)^{3/2}} dx
			-\int_0^\infty\frac{a(1+x^2)}{2(1+x^2)^{3/2}(a+x^2)^{3/2}} dx \\
			&=\left[\frac{x}{\sqrt{1+x^2}\sqrt{a+x^2}}\right]_0^\infty
			-\int_0^\infty\frac{a(1+x^2)}{2(1+x^2)^{3/2}(a+x^2)^{3/2}} dx  \\
			&=-\int_0^\infty\frac{a(1+x^2)}{2(1+x^2)^{3/2}(a+x^2)^{3/2}}  dx<0 .
		\end{align}
		Hence, $I(a)\le I(0)=2$. Therefore, combining \eqref{eq:hyp-Q-reduction}
		and \eqref{eq:hyp-kernel-second-transform}, we obtain
		\begin{equation}\label{eq:hyp-standard-Q-explicit-bound}
			\left| Q_\nu'(\cosh r)\right|\sinh^2r
			\le2\sqrt{\mu}e^{r/2}.
		\end{equation}
		Finally, since $\mathcal Q_\nu=\operatorname{Re}\{ Q_\nu\}$,
		we have
		\begin{align}
			\left|\mathcal Q_\nu'(\cosh r)\right|\sinh^2r
			\le
			\left| Q_\nu'(\cosh r)\right|\sinh^2r\le 2\sqrt{\mu}e^{r/2}.
		\end{align}
		This completes the proof.
	\end{proof}

	\begin{theorem}\label{thm:hyperbolic-absolute-const2}
		Let $\Omega\subset\Hyp$ be a bounded convex domain. Then
		\begin{equation}
			\mathfrak C(\Omega) < 11.2.
		\end{equation}
	\end{theorem}
	
	\begin{proof}
		Let $D = \diam_{\Hyp}(\Omega)$, $\mu = \mu_2(\Omega)$,
		and $\nu$ be the complex constant that satisfies $\nu(\nu+1)=-\mu$.
		If the second Neumann eigenfunction has no interior critical point, then
		$\mathfrak C(\Omega)=1$  is trivial. Hence,
		by Remark \ref{rem:hyperbolic-hatcher-recovery} and Theorem \ref{thm:hyp},
		without loss of generality, we may assume that $\mu> 1/4$ and $\mu D^2\ge j_{1,1}^2$
		throughout the proof.
		We divide the proof into four cases.
		
		\medskip
		\noindent
		\textbf{Case 1: Large area.}
		
		Set $A_0=169\pi/14$.
		We claim that every second Neumann eigenfunction has no interior critical point when $\Area(\Omega)\ge A_0$.
		Indeed, by Corollary 1.2 of \cite{Hatcher2026Arxiv}, it suffices to
		choose an explicit area larger than the area of the hyperbolic disk whose
		second Neumann eigenvalue is $1/4$.
		
		Let $R_0=\log 14$.
		On the hyperbolic disk $B_{R_0}$, we consider the test function
		$\phi(r,\theta)=\tanh\frac r2\cos\theta$.
		Its Rayleigh quotient is
		\begin{equation}\label{eq:hyp-disk-trial-Rayleigh}
			\mathcal R[\phi]
			=
			\frac{\tanh^2(R_0/2)}
			{\cosh R_0-1-4\ln\cosh(R_0/2)}.
		\end{equation}
		Since
		\[
		\tanh\frac{R_0}{2}=\frac{13}{15},
		\quad
		\cosh R_0=\frac{197}{28},
		\quad
		\cosh\frac{R_0}{2}=\frac{15}{\sqrt{56}}<e^{3/4}.
		\]
		We obtain
		\[
		\mathcal R[\phi]<\frac14.
		\]
		The second Neumann eigenvalue of a hyperbolic disk is strictly decreasing with its radius; see \cite[Theorem 2]{LangfordLaugesen2023JST}.  Hence the radius $R_*$ for which $\mu_2(B_{R_*})=1/4$ satisfies $R_*<R_0$.
		Since $\Area(B_{R_0})=2\pi(\cosh R_0-1)=A_0$,
		Corollary 1.2 of \cite{Hatcher2026Arxiv} implies that
		the second Neumann eigenfunction has no interior critical point,
		so $\mathfrak C(\Omega)=1$.

		Hence it remains to consider the case of $\Area(\Omega)<A_0$.

		\noindent
		\textbf{Case 2: $\frac14<\mu\le 2$.}
		
		We write $\nu=-\frac12+i\rho$, $\rho=\sqrt{\mu-\frac14}$.
		Taking $c=0$ in \eqref{eq:hyp-green-identy} and applying Lemma \ref{lem:hyp-uniform-conical-kernel}, we obtain
		\begin{align}\label{eq:hyp-low-int-bound}
			\mathfrak C(\Omega)
			\le\frac{1}{2\pi}\int_0^{2\pi}\left|\mathcal Q_\nu'(\cosh r(\theta))\right|\sinh^2r(\theta) d\theta
			\le\frac{\sqrt{\mu}}{\pi}\int_0^{2\pi}e^{r(\theta)/2}d\theta.
		\end{align}
		The area formula gives
		\begin{equation}\label{eq:hyp-radial-area-identity}
			\operatorname{Area}(\Omega)
			=\int_0^{2\pi}\int_0^{r(\theta)}\sinh r\,dr\,d\theta
			=\int_0^{2\pi}\bigl(\cosh r(\theta)-1\bigr)\,d\theta.
		\end{equation}
		Since $e^r\le2\cosh r$ for $r\ge0$,  \eqref{eq:hyp-radial-area-identity} yields
		\[
		\int_0^{2\pi}e^{r(\theta)} d\theta
		\le
		2\bigl(\Area(\Omega)+2\pi\bigr).
		\]
		Therefore, by the Cauchy-Schwarz inequality, we have
		\begin{align}\label{eq:hyp-cauchy-area}
			\int_0^{2\pi}e^{r(\theta)/2} d\theta
			\leq(2\pi)^{1/2}
			\left(\int_0^{2\pi}e^{r(\theta)} d\theta\right)^{1/2} \leq2\sqrt{\pi\bigl(\Area(\Omega)+2\pi\bigr)}.
		\end{align}
		Consequently, \eqref{eq:hyp-low-int-bound} and \eqref{eq:hyp-cauchy-area} give
		\begin{align}\label{eq:hyp-low-area}
			\mathfrak C(\Omega)\le 2\sqrt{\frac{\mu(\operatorname{Area}(\Omega) +2\pi)}{\pi}}.
		\end{align}
		Since $\mu\le 2$, together with $\Area(\Omega)<A_0$, we have
		\begin{align}
			\mathfrak C(\Omega)\le 2\sqrt{2\left(\frac{169}{14}+2\right)}
			=2\sqrt{\frac{197}{7}}<11.
		\end{align}

		\medskip
		\noindent
		\textbf{Case 3: $2<\mu\le 12$.}
		
		Let $\mathcal B_R\subset\Hyp$ be the geodesic disk such that $\Area(\mathcal B_R)=\Area(\Omega)$.
		Applying Lemma \ref{lem:hyperbolic-diameter-upper2} to the geodesic disk $\mathcal B_R$, whose
		diameter is $2R$, yields
		\begin{align}\label{eq:hyp-mu-isoper}
			\mu\le \mu_2(B_R)\le\frac13+\frac{j_{0,1}^2}{R^2}.
		\end{align}
		where the first inequality is the sharp isoperimetric inequality for Neumann eigenvalue in
		hyperbolic space \cite{Ashbaugh1995JLMS}. Consequently,
		\begin{equation}\label{eq:hyp-radius-from-mu}
			R\le\frac{j_{0,1}}{\sqrt{\mu-1/3}}.
		\end{equation}
		By the area formula \eqref{eq:hyp-radial-area-identity}, the area of the geodesic disk is given by $\operatorname{Area}(\mathcal B_R)=2\pi\cosh R-2\pi$. Therefore, together with \eqref{eq:hyp-low-area}
		and \eqref{eq:hyp-radius-from-mu}, we have
		\begin{align}\label{eq:hyp-hot-constant-mu}
			\mathfrak C(\Omega)\le 2\sqrt{\frac{\mu(\operatorname{Area}(\Omega) +2\pi)}{\pi}}
			\le 2\sqrt{2\mu \cosh \left(\frac{j_{0,1}}{\sqrt{\mu-1/3}} \right)}
		\end{align}
		If $2< \mu\le 3$, we simply use \eqref{eq:hyp-hot-constant-mu} to obtain
		\begin{align}
			\mathfrak C(\Omega)\le 2 \sqrt{6\cosh \left(\frac{j_{0,1}}{\sqrt{2-1/3}} \right)}<9.
		\end{align}
		If $3< \mu\le 12$. Define
		\begin{align}
			F(\mu)
			=\mu\cosh \left(\frac{j_{0,1}}{\sqrt{\mu-1/3}}\right),\quad x(\mu)=\frac{j_{0,1}}{\sqrt{\mu-1/3}}.
		\end{align}
		We now establish the monotonicity of $F(\mu)$. A direct differentiation gives
		\begin{align}
			\frac{F'(\mu)}{F(\mu)}=\frac1\mu-\frac{x(\mu)\tanh x(\mu)}{2(\mu-1/3)}.
		\end{align}
		Since $\tanh x<1$ for $x>0$, and $x(\mu)<\frac{j_{0,1}}{\sqrt{3-1/3}}<1.48<\frac{2(\mu-1/3)}{\mu}$
		holds for $\mu>3$, it follows that $F$ is increasing on $[3,\infty)$.
		Therefore \eqref{eq:hyp-hot-constant-mu} gives
		\begin{align}
			\mathfrak C(\Omega)
			\le2\sqrt{2\mu \cosh \left(\frac{j_{0,1}}{\sqrt{\mu-1/3}} \right)}
			<2\sqrt{2\cdot 12\cosh\left(\frac{j_{0,1}}{\sqrt{12-1/3}} \right)}<11.
		\end{align}

		\medskip
		\noindent
		\textbf{Case 4: $\mu> 12$.}
		
		Let $t=\sqrt{\mu}r$, $g(r,\nu)=-\sinh^2r\mathcal Q_\nu'(\cosh r)$ and $\displaystyle G_\nu(t)=g (\frac{t}{\sqrt{\mu}}, \nu )$.
		Using the Legendre differential equation, a direct computation shows that $G_\nu(t)$ satisfies
		\begin{align}
			G''_\nu(t)-\frac{1}{\sqrt{\mu}}\coth\left(\frac{t}{\sqrt{\mu}}\right)G'_\nu(t)+G_\nu(t)=0.
		\end{align}
		Equivalently,
		\begin{equation}\label{eq:hyp-G-bessel-perturbation}
			G''_\nu-\frac1tG'_\nu+G_\nu=\varepsilon_\mu(t)G'_\nu,
		\end{equation}
		where
		\[
		\varepsilon_\mu(t)
		=
		\frac1{\sqrt{\mu}}
		\coth\left(\frac{t}{\sqrt{\mu}}\right)
		-\frac1t.
		\]
		Using the elementary inequality $\displaystyle 0<\coth x-\frac1x\le\frac{x}{3}$ for $x>0$, we have
		\begin{equation}\label{eq:hyp-epsilon-bound}
			0\le\varepsilon_\mu(t)\le\frac{t}{3\mu}.
		\end{equation}
		Moreover, Lemma \ref{lem:hyperbolic-diameter-upper2} yields
		\[
		D^2\left(\mu-\frac13\right)
		\le4j_{0,1}^2.
		\]
		Hence, since $0<r<D$, we have
		\begin{align}
			t=\sqrt{\mu}r<2j_{0,1}\sqrt{\frac{\mu}{\mu-1/3}}<5.
		\end{align}
		
		Let $U(t)=tJ_1(t)$ and $\displaystyle V(t)=-\frac{\pi}{2}tY_1(t)$,
		they are two linearly independent solutions of the homogeneous equation $H''-\frac1t H'+H=0$.
		By the local expansion formula of the Olver function $\mathbf{Q}_\nu$ as $x\to 1^+$ \cite[(14.8.9)]{DLMF},
		together with the normalization relation \cite[(14.3.10)]{DLMF}, we have
		\begin{align}
			Q_\nu=\frac{1}{2}\ln(\frac{2}{x-1})-\gamma -\psi(\nu+1)+O((x-1)\ln(x-1)).
		\end{align}
		Together with $\cosh r-1=r^2/2+O(r^4)$, and the definition $\mathcal Q_\nu=\operatorname{Re}\{ Q_\nu\}$,
		proceeding as in Lemma \ref{lem:G-nu-asymptotic}, with the extra step of taking the modulus,
		we may choose $c_\nu=\frac12\ln\mu-\operatorname{Re} \psi\left(\nu+1\right)$
		such that $H_\nu(t)=V(t)+c_\nu U(t)$ has the same $t^2$-order behavior as $G_\nu(t)$.

		We next derive an estimate for $c_\nu$. Let $z=\nu+1=\frac{1}{2}+i\rho$, 
		by the integral representation of the digamma function (\cite[(5.9.15)]{DLMF}), we have
		\begin{align}\label{eq:digamma-inte}
			\psi(z)=\ln z-\frac{1}{2z}-2\int_{0}^{\infty}\frac{t}{(t^2+z^2)(e^{2\pi t} -1)}dt
		\end{align}
		Note that $|z|^2=\mu$, $\operatorname{Re}\frac{1}{2z}=\frac{1}{4\mu}$,
		taking the real part of \eqref{eq:digamma-inte}, we obtain
		\begin{align}
			|c_\nu|\le& \frac{1}{4\mu}+2\int_{0}^{\infty}\frac{t}{|t^2+z^2|(e^{2\pi t} -1)}dt \\
			\le & \frac{1}{4\mu}+\frac{2}{\rho}\int_{0}^{\infty}\frac{t}{e^{2\pi t} -1}dt
			=\frac1{4\mu}+\frac1{12\rho},
		\end{align}
		where the last inequality follows from $|t^2+z^2|\ge\rho$.
		Therefore, $\mu\ge 12$ gives $|c_\nu|<\frac1{22}$.

		As in the proof of Theorem \ref{thm:sphere-universal-const-4}, by using Lemma \ref{lem:U-V-estimate},
		we obtain the following estimates for $H_\nu$ and Green kernel $K$:
		\begin{equation}
			\begin{aligned}
				&	|H_\nu(t)|, |H_\nu'(t)|\le \frac{27}{10}+\frac1{22}\frac{19}{10}=\frac{613}{220},\\
				&	|K(t,s)|, |\partial_tK(t,s)|\le\frac{513}{50s}.
			\end{aligned}
		\end{equation}
		Since $H_\nu(t)$ has the same $t^2$-order behavior as $G_\nu(t)$,
		Lemma \ref{lem:bessel-apro}  gives
		\begin{equation}\label{eq:general-green-representation2}
			G_\nu(t)=H_\nu(t)+\int_0^t K(t,s)\varepsilon_\mu(s)G'_\nu(s) ds.
		\end{equation}
		Differentiating \eqref{eq:general-green-representation2}, combining the estimates for $H_\nu$, $K$ and \eqref{eq:hyp-epsilon-bound}, we have
		\begin{align}
			|G'_\nu(t)|\le \frac{613}{220} +\frac{171}{50\mu} \int_0^t|G_\nu'(s)|ds.
		\end{align}
		Gronwall's inequality implies $\displaystyle |G_\nu'(t)|\le\frac{613}{220}\exp\left(\frac{171t}{50\mu}\right)$, 
		substituting this into \eqref{eq:general-green-representation2} gives
		\begin{equation}\label{eq:hyp-G-final-bound}
			|G_\nu(t)|\le\frac{613}{220}\exp\left(\frac{171t}{50\mu}\right).
		\end{equation}
		Using $t<2j_{0,1}\sqrt{\frac{\mu}{\mu-1/3}}$ and $\mu\ge 12$, we have
		\begin{align}
			|G_\nu(t)|\le\frac{613}{220}\exp \left(\frac{171\cdot 2\cdot j_{0,1}}{50\cdot12 }\sqrt{\frac{12}{12-1/3}}\right)
			<11.2.
		\end{align}
		Therefore Proposition \ref{thm:hot-spots-constant-hyp} gives $\mathfrak C(\Omega)<11.2$ in this range.

		Combining the four cases, we conclude that
		\[	\mathfrak C(\Omega) < 11.2\]
		This completes the proof.
	\end{proof}

\end{document}